\documentclass[11pt]{article}  
\usepackage{amsmath}
\usepackage{amssymb}
\usepackage{theorem}
\usepackage{euscript}
\usepackage{pstricks}
\usepackage{authblk}
\usepackage{verbatim}
\usepackage{graphics}
\usepackage{graphicx}
\usepackage{color}

\usepackage{hyperref}
\hypersetup
{
	colorlinks,
	citecolor=black,
	filecolor=black,
	linkcolor=black,
	urlcolor=blue
}
\hypersetup{linktocpage}
\newtheorem{theorem}{Theorem}[section]
\newtheorem{lemma}[theorem]{Lemma}
\newtheorem{corollary}[theorem]{Corollary}
\numberwithin{equation}{section}

\theoremstyle{definition}
 
\theoremstyle{remark}

\newcommand{\brac}[1]{\left(#1\right)}

\newcommand{\brab}[1]{\left\{#1\right\}}

\newcommand{\bk}{{\boldsymbol{k}}}

\newcommand{\br}{{\boldsymbol{r}}}
\newcommand{\bs}{{\boldsymbol{s}}}

\newcommand{\bx}{{\boldsymbol{x}}}
\newcommand{\bE}{{\boldsymbol{E}}}

\newcommand{\by}{{\boldsymbol{y}}}

\newcommand{\bW}{{\boldsymbol{W}}}

\newcommand{\rd}{{\rm d}} 
\newcommand{\Int}{{\rm Int}}

\newcommand{\sign}{{\rm sign\,}} 
\def\II{\mathbb I}

\def\ZZd{{\mathbb Z}^d}
\def\IId{{\mathbb I}^d}
\def\ZZ{{\mathbb Z}}

\def\RR{{\mathbb R}}
\def\RRd{{\mathbb R}^d}

\def\NN{{\mathbb N}}
\def\NNd{{\NN}^d}

\def\II{{\mathbb I}}

\def\NN{{\mathbb N}}
\def\RR{{\mathbb R}}

\def\FF{{\mathcal F}}

\def\IId{{\mathbb I}^d}

\def\NNd{{\mathbb N}^d}
\def\RRd{{\mathbb R}^d}

\def\ZZd{{\mathbb Z}^d}

\def\Hh{{\mathcal H}}

\def\II{{\mathbb I}}

\def\ZZ{{\mathbb Z}}
\def\NN{{\mathbb N}}
\def\RR{{\mathbb R}}

\def\FF{{\mathbb F}}

\def\IId{{\mathbb I}^d}

\def\NNd{{\mathbb N}^d}
\def\RRd{{\mathbb R}^d}

\def\supp{\operatorname{supp}}

\def\Wpmix{W^\alpha_p(\IId)}

\def\Wpgamma{W^\alpha_p(\RRd,\gamma)}
\def\Wap{W^\alpha_p}
\def\Wa{W^\alpha_2(\RRd,\gamma)}
\def\Lqgamma{L_q(\RRd,\gamma)}

\def\BWpmix{\bW^\alpha_p(\IId)}

\def\BWpgamma{\bW^\alpha_p(\RRd,\gamma)}

\newcommand{\norm}[2]{\left\|{#1}\right\|_{#2}}

\title{\sffamily {Optimal numerical integration    and approximation of functions
on $\RRd$ equipped with Gaussian measure}}

\author[a]{Dinh D\~ung}
\affil[a]{Information Technology Institute, Vietnam National University, Hanoi
	\protect\\
	144 Xuan Thuy, Cau Giay, Hanoi, Vietnam
	\protect\\
	Email: dinhzung@gmail.com}

\author[b]{Van Kien Nguyen}
\affil[b]{Department of Mathematical Analysis, University of Transport and Communications
	\protect\\	No.3 Cau Giay Street, Lang Thuong Ward, Dong Da District,
	Hanoi, Vietnam
	\protect\\
	Email: kiennv@utc.edu.vn}

\date{\today}
\begin{document}
\maketitle

\begin{abstract}
 We investigate the numerical approximation of integrals over $\mathbb{R}^d$ equipped with the standard Gaussian measure $\gamma$  for  {integrands}  belonging to the Gaussian-weighted  Sobolev spaces $W^\alpha_p(\mathbb{R}^d, \gamma)$ of  mixed smoothness $\alpha \in \mathbb{N}$ for $1 < p < \infty$. We prove the asymptotic order  of the convergence  of  optimal quadratures based on $n$ integration nodes and propose a novel method for constructing asymptotically optimal quadratures. As for related problems,  we establish by a similar technique  the  asymptotic order of the linear,  Kolmogorov   and sampling $n$-widths in the Gaussian-weighted space $L_q(\mathbb{R}^d, \gamma)$ of the unit ball of $W^\alpha_p(\mathbb{R}^d, \gamma)$ for $1 \leq q < p < \infty$ and  $q=p=2$. 
	
	\medskip
	\noindent
	{\bf Keywords and Phrases}: Multivariate numerical integration; Quadrature; Multivariate approximation; Gaussian-weighted Sobolev space of mixed smoothness; $n$-Widths; Asymptotic order of convergence. 
	
	\medskip
	\noindent
	{\bf MSC (2020)}:   65D30; 65D32; 41A25; 41A46.
	
\end{abstract}

\section{Introduction}
 \label{Introduction}

We investigate numerical approximation of  integrals 
\begin{equation} \label{I(f)}
I(f):=\int_{\RRd} f(\bx) \, \gamma(\rd\bx) = \int_{\RRd} f(\bx) g(\bx) \, \rd\bx
\end{equation}
for functions $f$ belonging  to the {Gaussian-weighted}  Sobolev spaces 
$W^\alpha_p(\mathbb{R}^d, \gamma)$ of  mixed smoothness $\alpha \in \mathbb{N}$ for $1 < p < \infty$ (see Section \ref{Numerical integration} for the definition),
where $\gamma (\rd \bx) = g(\bx) \rd\bx$  is the $d$-dimensional standard Gaussian measure on $\RRd$  with the density 
$$
g(\bx):=(2\pi)^{-d/2} \exp\brac{-|\bx|^2/2},\ \ \bx\in \RRd.
$$
To approximate this integral we use a (linear) quadrature  defined by
\begin{equation} \label{I_n(f)-introduction}
	I_n(f): = \sum_{i=1}^n \lambda_i f(\bx_i) 
\end{equation}
with the convention $I_0(f) = 0$, where $\{\bx_1,\ldots,\bx_n\}\subset \RRd$  are given integration nodes and $(\lambda_1,\ldots,\lambda_n)$ the integration weights. For convenience, we assume that some of the integration nodes may coincide.
Let $\BWpgamma$ be the unit ball of $\Wpgamma$. The optimality  of  quadratures  for  
$\BWpgamma$ is measured by the quantity
\begin{equation} \label{Int_n}
\Int_n(\BWpgamma) :=\inf_{I_n}\sup_{f\in \BWpgamma}|I(f)-I_n(f)|.
\end{equation}

We are interested in the asymptotic order  of this quantity when $n \to \infty$, as well as in constructing asymptotically optimal quadratures. We do not investigate the dependence on the dimension and the problem of tractability. The problem of multivariate numerical integration \eqref{I(f)}--\eqref{I_n(f)-introduction} has been studied in \cite{IKLP2015, IL2015,DILP18} for functions in certain Hermite spaces, in particular, the space $\Hh_{d,\alpha}$ in \cite{DILP18}  which coincides with $W^\alpha_2(\mathbb{R}^d, \gamma)$ in terms of norm equivalence.
So far the best result on this problem is
	\begin{equation*}\label{DILP18}
n^{-\alpha} (\log n)^{\frac{d-1}{2}} 
\ll	
\Int_n\big( \bW^\alpha_2(\RRd, \gamma)\big)  
\ll 
n^{-\alpha} (\log n)^{\frac{d(2\alpha + 3)}{4} - \frac{1}{2}},
\end{equation*}
which has been proven in \cite{DILP18}. Moreover,  the upper bound is achieved by  a translated and scaled quasi-Monte Carlo (QMC) quadrature based on Dick's higher order digital nets. We note the related work \cite{KPPW2020} which studied weighted integration via a change of variables for functions on $\RRd$ from non-weighted spaces of mixed smoothness.

The aim of this paper is to prove the asymptotic order  of  $\Int_n\big( \bW^\alpha_p(\RRd, \gamma)\big)$.
Let us briefly describe the main results. 

For  $\alpha\in \NN$ and $1<p<\infty$, we construct  an  asymptotically optimal quadrature  $I_n^\gamma$ of the form \eqref{I_n(f)-introduction} which  gives the  asymptotic order of the  convergence
\begin{equation} 	\label{AsympQuadrature}
	\sup_{f\in \BWpgamma} \bigg|\int_{\RRd}f(\bx) \gamma(\rd\bx) - I_n^\gamma(f)\bigg| 
\asymp
	\Int_n\big(\BWpgamma\big) 
	\asymp
	n^{-\alpha} (\log n)^{\frac{d-1}{2}}.
\end{equation}
In constructing $I_n^\gamma$, we  propose a novel method assembling an asymptotically optimal quadrature for the related Sobolev spaces on the unit $d$-cube to the  integer-shifted $d$-cubes which cover $\RRd$. The asymptotically optimal quadrature  $I_n^\gamma$ 
is based on very sparse  integration nodes contained in a $d$-ball of radius
$\sqrt{\log n}$.

As for related problems with a similar approach,  we establish the  asymptotic orders of linear $n$-widths  $\lambda_n$, Kolmogorov  $n$-widths $d_n$, and  sampling $n$-widths $\varrho_n$ of the set $\BWpgamma$ in the Gaussian-weighted space $\Lqgamma$  (see Section \ref{Approximation} for definitions).
For $\alpha\in \NN$ and $1\le q<p<\infty$  we prove that
\begin{equation}\label{widths:p>q-introduction}
	\lambda_n
	\asymp		
	d_n
	\asymp 
	n^{-\alpha} (\log n)^{(d-1)\alpha},
\end{equation}
and with the additional condition $q=2$,
\begin{align}\label{sampling-widths:q=2}
	\varrho_n
	\asymp 
	n^{-\alpha} (\log n)^{(d-1)\alpha}.
\end{align}		
For  $\alpha\in \NN$ and $q=p=2$, we prove that
	\begin{equation}\label{widths:p=q=2-introduction}
	\lambda_n 
	=
	d_n
	\asymp 
	n^{-\frac{\alpha}{2}} (\log n)^{\frac{(d-1)\alpha}{2}}, 
\end{equation}		
and with the additional condition $\alpha \ge 2$,
	\begin{equation}\label{sampling-widths:p=q=2-introduction}
	\varrho_n
	\asymp 
	n^{-\frac{\alpha}{2}} (\log n)^{\frac{(d-1)\alpha}{2}}.
\end{equation}
The asymptotic orders \eqref{widths:p>q-introduction}--\eqref{sampling-widths:p=q=2-introduction} show very different approximation results  between the cases $q < p$ and  $q=p=2$. We conjecture that the asymptotic orders \eqref{widths:p=q=2-introduction} and \eqref{sampling-widths:p=q=2-introduction} still hold true for $p,q$ with the restrictions $p = q \not= 2$ and $1 <p< \infty$. The case $1\le p < q <\infty$ of these $n$-widths is excluded from the consideration caused by the natural reason that in this case we do not have  a continuous embedding of $\Wpgamma$ into $L_q(\RRd,\gamma)$. For example, the function $f(\bx)=\prod_{i=1}^d\big(1+x_i^2\big)^{-m}\exp \big(|\bx|^2/(2p)\big)$ belongs to $\Wpgamma$ if $m >1/2+\alpha$. However, this function does not belong to $L_q(\RRd,\gamma)$ when $q >p$.

The paper is organized as follows. In Section \ref{Numerical integration}, we prove the asymptotic order 
 of $\Int_n\big(\BWpgamma\big)$ and construct asymptotically optimal quadratures.  Section \ref{Approximation} is devoted to the proof of the asymptotic order of linear $n$-widths  $\lambda_n$ and Kolmogorov $n$-widths $d_n$  for the cases $q < p$ and  $q=p=2$ and the construction of asymptotically optimal linear approximations. In this section we also give asymptotic order of sampling $n$-widths for the cases $q=2 < p$ and $q=p=2$. In Section \ref{Sec-4}, we  illustrate our integration nodes in comparison with those used in \cite{DILP18} and give a numerical test for the results obtained in Section \ref{Numerical integration}.
\\


\noindent
{\bf Notation.} We write $\RR_1:= \brab{x \in \RR: x \ge 1}$.
For a Banach space $E$, denote by the bold symbol $\bE$ the unit ball in $E$. The letter $d$ is always reserved for
the underlying dimension of $\RR^d$, $\NN^d$, etc. Vectors in $\RRd$  are denoted by boldface
letters. For $\bx \in \RR^d$, $x_i$ denotes the $i$th coordinate, i.e., $\bx := (x_1,\ldots, x_d)$.  If $ 1\le p\leq \infty$, we write
$|\bx|_p := \big(\sum_{i=1}^d |x_i|^p\big)^{1/p}$ with the usual modification when $p=\infty$. When $p=2$ we simply write $|\bx|$. 
For the quantities $A_n$ and $B_n$ depending on 
$n$ in an index set $J$  
we write  $A_n \ll B_n$  
if there exists some constant $C >0$ independent of $n$ such that 
$A_n \leq CB_n$ for all $n \in J$, and  
$A_n \asymp B_n$ if $A_n  \ll B_n $
and $B_n  \ll A_n $. General positive constants or positive constants depending on parameters $\alpha, d,\ldots$ are denoted by $C$ or $C_{\alpha,d,\ldots}$, respectively. Values of constants $C$ and  $C_{\alpha,d}$ in general, are not specified except in the cases when they are precisely given, and may be different in various places. Denote by $|G|$ the cardinality of the finite set $G$. 


	\section{Numerical integration}
\label{Numerical integration}

In this section, based on a quadrature on  the  $d$-cube $\IId:=\big[-\frac{1}{2}, \frac{1}{2}\big]^d$  for numerical integration of functions from classical Sobolev spaces of mixed smoothness on $\IId$, by assembling we construct  a quadrature on $\RRd$ for numerical integration of functions from $\gamma$-weighted Sobolev spaces $\Wpgamma$ which preserves   the  convergence rate.  As a consequence, we prove the asymptotic order of 
 $\Int_n\big(\BWpgamma\big)$.

\subsection{Assembling quadratures}
\label{Subsec-AssemblingQuadratures}
   We first introduce $\gamma$-weighted Sobolev spaces of mixed smoothness.	Let  $1\leq p<\infty$ and $\Omega$ be a Lebesgue measurable set on $\RRd$. 
	We define the $\gamma$-weighted space  $L_p(\Omega,\gamma)$ to be the set of all functions $f$ on $\Omega$ such that the norm
	$$
	\|f\|_{L_p(\Omega,\gamma)} : = \bigg( \int_\Omega |f(\bx)|^p \gamma(\rd \bx)\bigg)^{1/p}
	=
	\bigg( \int_\Omega |f(\bx)|^p g(\bx) \rd \bx\bigg)^{1/p} \ <  \ \infty. 
	$$
	For $\alpha \in \NN$, we define the $\gamma$-weighted  space $\Wap(\Omega,\gamma)$ to be the normed space of all functions $f\in L_p(\Omega,\gamma)$ such that the weak (generalized) partial derivative $D^\br f$ of order $\br$  belongs to $L_p(\Omega,\gamma)$ for all $\br\in \NN_0^d$ satisfying $|\br|_\infty\leq \alpha$. The norm of a  function $f$ in this space 
	is defined by
	\begin{align} \label{W-Omega}
		\|f\|_{\Wap(\Omega,\gamma)}: = \Bigg(\sum_{|\br|_\infty \leq \alpha} \|D^\br f\|_{L_p(\Omega,\gamma)}^p\Bigg)^{1/p}.
	\end{align}
	The space $\Wap(\Omega)$ is defined  as the classical Sobolev space by replacing $L_p(\Omega,\gamma)$ with $L_p(\Omega)$ in \eqref{W-Omega}, where as usual,   $L_p(\Omega)$ denotes the Lebesgue space of functions on $\Omega$ equipped with the usual $p$-integral norm. For technical  convenience we use the conventions 
	$\Int_n := \Int_{\lfloor n \rfloor}$ and $I_n := I_{\lfloor n \rfloor}$ for $n \in \RR_1$.
	
 For numerical approximation of integrals 
	$
	I^\Omega(f):=\int_\Omega f(\bx) \rd\bx
	$
	over the set $\Omega$, we need natural modifications $I_n^\Omega(f)$ 	for functions $f$ on $\Omega$,  and $\Int_n^\Omega(F)$  for  a set $F$ of  functions on  $\Omega$, of the definitions \eqref{I_n(f)-introduction} and \eqref{Int_n}. For simplicity we will drop $\Omega$ from these notations if there is no misunderstanding.


\label{Quadratures on sparse-digital-nets }
Let $\alpha\in \NN$, $1<p<\infty$ and $a >0$, $b \ge 0$.  	
Assume that for the quadrature
\begin{equation}\label{I_m(f)}
	I_m(f): = \sum_{i=1}^m \lambda_i f(\bx_i), \ \ \{\bx_1,\ldots,\bx_m\}\subset \IId,
\end{equation}
holds the convergence rate
\begin{equation}\label{IntError-a,b}
	\bigg|\int_{\IId} f(\bx) \rd \bx  - I_m(f)\bigg| \leq C m^{-a} (\log m)^b \|f\|_{\Wpmix}, 
	\ \  f\in \Wpmix.
\end{equation}
Then based on  $I_m$, we will construct  a quadrature  on $\RRd$  which approximates the integral $I(f)$ with the same convergence rate for 
$f \in \Wpgamma$. 

Our strategy is as follows.  
The integral $I(f)$ can be represented as the sum of component integrals over the  integer-shifted $d$-cubes $\IId_{\bk}$ by 
\begin{align} \label{Int_RRd}
I(f)=	 \sum_{\bk \in \ZZd}\int_{\IId_\bk}f_\bk(\bx)g_\bk(\bx)\rd \bx,
\end{align}
where  for $\bk \in \ZZd$, $\IId_\bk:=\bk+\IId$ 
and for a function $f$  on $\RRd$,  
$f_\bk $ denotes the restriction of $f$ to $\IId_\bk$.
For a given $n \in \RR_1$, we take ``shifted" quadratures $I_{n_\bk}$  of the form \eqref{I_m(f)} for  approximating the component integrals in the sum in \eqref{Int_RRd}. 
The integration nodes  in  
$I_{n_\bk}$, $\bk \in \ZZd$, are taken so that they become sparser as $|\bk|$ gets larger and
$$
\sum_{\bk \in \ZZd} \lfloor n_\bk \rfloor \le n.
$$
In the next step, we  ``assemble" these shifted integration nodes to form a quadrature  $I_n^\gamma$ for approximating $I(f)$. Let us describe this construction in detail.

It is clear that if $f\in \Wpgamma$, then 
$
f_\bk(\cdot+\bk) \in \Wpmix,
$
and
\begin{equation}\label{eq:b4}
\begin{split}
	\|f_\bk(\cdot+\bk)\|_{\Wpmix}&=\Bigg(\sum_{|\br|_\infty \leq \alpha} \|D^\br f_\bk(\cdot+\bk)\|_{L_p(\IId)}^p\Bigg)^{1/p}
	\\
&=\Bigg(\sum_{|\br|_\infty \leq \alpha} \|D^\br f_\bk\|_{L_p(\IId_\bk)}^p\Bigg)^{1/p}
\\ 
&=\Bigg(\sum_{|\br|_\infty \leq \alpha} (2\pi)^{d/2} \int_{\IId_\bk}e^{\frac{|\bx|^2}{2}}|D^\br f_\bk(\bx)| ^pg(\bx)\rd\bx \Bigg)^{1/p}.
\end{split}
\end{equation}
When $\bx \in \IId_{\bk}$ we have 
$e^{\frac{|\bx|^2}{2}}\leq e^{\frac{|\bk + (\sign \bk)/2 |^2}{2}}$,  where $\sign \bk:= \brac{\sign k_1, \ldots , \sign k_d}$ and $\sign x := 1$ if $x \ge 0$, and $\sign x := -1$ otherwise for $x \in \RR$. Therefore,
\begin{equation}\label{eq:norm-fwid}
	\|f_\bk(\cdot+\bk)\|_{\Wpmix}
	\leq  (2\pi)^{\frac{d}{2p}}e^{\frac{|\bk + (\sign \bk)/2 |^2}{2p}}\|f\|_{\Wpgamma}.
\end{equation}
We have
$$\|g_{\bk}(\cdot+\bk)\|_{\Wpmix}=\Bigg(\sum_{|\br|_\infty \leq \alpha} \|D^\br g\|_{L_p(\IId_\bk)}^p\Bigg)^{1/p}.$$
A direct computation shows that for $\br \in \NN_0^d$ we have $D^\br g(\bx)=P_\br(\bx)g(\bx)$ where $P_\br(\bx)$ is a polynomial of order $|\br|_1$ of $\bx$. Moreover, we have $-|x|^2\leq \frac{1}{2}-|k-(\sign k)/2|^2 $ for $x\in [-\frac{1}{2},\frac{1}{2}]+k,\ k\in \ZZ$. Therefore for $\bx \in \IId_\bk$ we get
$$
|D^\br g(\bx)|=\Big|(2\pi)^{-d/2}P_\br(\bx)e^{-\frac{|\bx|^2}{2}}\Big| \leq Ce^{-\frac{|\bx|^2}{2\tau'}} \leq  Ce^{-\frac{|\bk - (\sign \bk)/2 |^2}{2\tau'}}\leq Ce^{-\frac{|\bk|^2}{2\tau}}
	$$
for some $\tau'$ and $\tau$ such that $1<\tau'<\tau<p<\infty $. This implies that
\begin{align} \label{g_bk}
	\|g_{\bk}(\cdot+\bk)\|_{\Wpmix}  \leq  Ce^{-\frac{|\bk|^2}{2\tau}}
\end{align}
with $C$ independent of $\bk \in \ZZd$. Since $\Wpmix$ is a multiplication algebra (see \cite[Theorem 3.16]{NgS17}), from \eqref{eq:norm-fwid} and \eqref{g_bk}  we have that 
\begin{align} \label{multipl-algebra1}
f_{\bk}(\cdot+\bk)g_{\bk}(\cdot+\bk)\in \Wpmix,
\end{align}
 and 
 \begin{equation}  
\begin{aligned} \label{multipl-algebra2}
	\|f_{\bk}(\cdot+\bk)g_{\bk}(\cdot+\bk)\|_{\Wpmix} 
	& 
	\leq C \|f_{\bk}(\cdot+\bk)\|_{\Wpmix}  \cdot \|g_{\bk}(\cdot+\bk)\|_{\Wpmix}
	\\
	& \leq C e^{\frac{|\bk + (\sign \bk)/2 |^2}{2p}-\frac{|\bk|^2}{2\tau}}\|f\|_{\Wpgamma}.
\end{aligned}
\end{equation}
For $1<\tau<p<\infty $, 
 we choose $\delta>0$ so that
\begin{equation}  \label{[tau]<e^{-delta k}1}
	\max \bigg\{e^{-\frac{|\bk - (\sign \bk)/2 |^2}{2}\big(1-\frac{1}{p}\big)},
	e^{\frac{|\bk + (\sign \bk)/2 |^2}{2p}-\frac{|\bk|^2}{2\tau}}\bigg\}
	\leq 
	C e^{-\delta |\bk|^2}
\end{equation}
for $\bk\in \ZZd$, and therefore,
\begin{align}  \label{f_{bk}<}
\|f_{\bk}(\cdot+\bk)g_{\bk}(\cdot+\bk)\|_{\Wpmix} 
	& \leq C e^{-\delta |\bk|^2}\|f\|_{\Wpgamma}, \qquad  \bk\in \ZZd.
\end{align}
We define	for $n\in \RR_1$,
\begin{align} \label{xi-int}	
	\xi_n =  \sqrt{\delta^{-1} 2 a(\log n)}\,,
\end{align}
and for $\bk \in \ZZd$,
\begin{align} \label{n_bk}
	n_{\bk}=
	\begin{cases}
		\varrho n  e^{-\frac{\delta}{2 a}|\bk|^2} &\text{if} \ |\bk|< \xi_n,
		\\
		0&\text{if}\  |\bk|\geq \xi_n,
	\end{cases}
\end{align}
where $\varrho := 2^{-d} \brac{1 - e^{-\frac{\delta}{2 a}}}^{d}$. We have
\begin{align} 	\label{<n2}
	\sum_{|\bk|< \xi_n}  n_\bk  \le n.
\end{align} 
Indeed,
\begin{equation*}
\begin{aligned}
	\sum_{|\bk| < \xi_n}n_\bk 
&	=
	\sum_{|\bk|< \xi_n} \varrho n  e^{-\frac{\delta}{2\alpha}|\bk|^2}
	\leq 
2^{d} \varrho	n  \sum_{s=0}^{\lfloor \xi_n \rfloor} \binom{s+d-1}{d-1}e^{-\frac{\delta}{2 a}s^2}
\\
&	\leq   2^{d}\varrho	n \sum_{s=0}^{\infty} \binom{s+d-1}{d-1}e^{-\frac{\delta}{2 a}s} \leq n, 
\end{aligned}
\end{equation*}
where in the last estimate we  used the well-known formula
\begin{equation}\label{eq-auxilary-01}
\sum_{j=0}^\infty x^j\binom{j+k}{k}=(1-x)^{-k-1}, \ k\in \NN_0, \ x\in (0,1).
\end{equation}
 We define
\begin{equation} \label{I_n^gamma}
I_n(f):=\sum_{|\bk|< \xi_n}I_{n_\bk}(f_{\bk}(\cdot+\bk)g_{\bk}(\cdot+\bk))
	= \sum_{|\bk|< \xi_n}\sum_{j=1}^{\lfloor n_\bk \rfloor} \lambda_j  f_{\bk}(\bx_j+\bk)g_{\bk}(\bx_j+\bk),
\end{equation}
or equivalently,
\begin{equation} \label{I_n^gamma2}
I_n(f):=	\sum_{|\bk|< \xi_n}\sum_{j=1}^{\lfloor n_\bk \rfloor} \lambda_{\bk,j} f(\bx_{\bk,j})
\end{equation}
as a quadrature for the approximate integration of $\gamma$-weighted functions $f$ on $\RRd$,
where $\bx_{\bk,j}:= \bx_j+\bk$ and $\lambda_{\bk,j}:= \lambda_j g_{\bk}(\bx_j+\bk)$ (here for simplicity, with an abuse of notation the dependence of integration nodes and weights on the quadratures  $I_{n_\bk}$ is omitted).  The  integration nodes of the quadrature $I_n$  are
\begin{equation} \label{int-nodes}
\{\bx_{\bk,j}: |\bk|< \xi_n, \, j=1,\ldots,\lfloor n_\bk \rfloor\}
\subset \RRd,
\end{equation}
and the  integration weights 
$$
(\lambda_{\bk,j}: |\bk|< \xi_n, \, j=1,\ldots,\lfloor n_\bk \rfloor).
$$
 Due to \eqref{<n2}, the number of integration nodes  is not greater than $n$.  From the definition we can see that the  integration nodes are contained in the ball of radius  $\xi_n^*:=\sqrt{d}/2 + \xi_n$, i.e., 
$\{\bx_{\bk,j}: |\bk|< \xi_n, \, j=1,\ldots,\lfloor n_\bk \rfloor\}
\subset B(\xi_n^*):= \brab{\bx \in \RRd:\, |\bx| \le \xi_n^*}$. The density of the integration nodes is exponentially decreasing in $|\bk|$ to zero from the origin of $\RRd$ to the boundary of the ball $B(\xi_n^*)$, and  the set of integration nodes is very sparse because of the choice of $n_{\bk}$ as in \eqref{n_bk}.

\begin{theorem} \label{thm:int-general}
	Let $\alpha\in \NN$, $1<p<\infty$ and $a >0$, $b \ge 0$.  
	Assume that for any $m \in \RR_1$, there is  a   quadrature  $I_m$  of the form \eqref{I_m(f)} satisfying \eqref{IntError-a,b}. Then for  the   quadrature  $I_{n}$ defined as in \eqref{I_n^gamma2} 
	we have
	\begin{equation} 	\label{IntError0}
		\bigg|\int_{\RRd}f(\bx) \gamma(\rd\bx) - I_n(f)\bigg| 
		\ll n^{-a}  (\log n)^b \|f\|_{\Wpgamma}, \ \  f \in \Wpgamma.
	\end{equation}
\end{theorem}
\begin{proof} 
Let $f\in \Wpgamma$ and $m\in \RR_1$. 	For the  quadrature $I_m$ for functions on $\IId$ in the assumption, from \eqref{IntError-a,b} and \eqref{f_{bk}<}  we have 
	\begin{equation} \label{IntError1}
		\begin{aligned}
			&\bigg|\int_{\IId} f_\bk(\bx+\bk)g_{\bk}(\bx+\bk)\rd \bx- I_{m}(f_{\bk}(\cdot+\bk)g_{\bk}(\cdot+\bk))\bigg| 
			\ll  m^{-a} (\log m)^b e^{-\delta |\bk|^2}\|f\|_{\Wpgamma}.
		\end{aligned}
	\end{equation}
	From \eqref{Int_RRd} and \eqref{I_n^gamma}
it follows that
	\begin{align*}
		\bigg|\int_{\RRd}f(\bx) \gamma(\rd\bx) - I_n(f)\bigg| 
		&\leq \sum_{|\bk|< \xi_n} \bigg|\int_{\IId_\bk} f_{\bk}(\bx)g_{\bk}(\bx)\rd \bx - 
		I_{n_\bk}(f_{\bk}(\cdot+\bk)g_{\bk}(\cdot+\bk))\bigg| 
		\\ &
		+ \sum_{|\bk|\geq \xi_n}\bigg|\int_{\IId_\bk}f_\bk(\bx)g_\bk(\bx)\rd \bx\bigg|.
	\end{align*}
	For each term in the first sum by   \eqref{IntError1} we derive  the estimates
	\begin{align*}
		\bigg|\int_{\IId_\bk} f_{\bk}(\bx)g_{\bk}(\bx)\rd \bx 
		& - I_{n_\bk}(f_{\bk}(\cdot+\bk)g_{\bk}(\cdot+\bk))\bigg| 
		\\& = \bigg|\int_{\IId} f_\bk(\bx+\bk)g_{\bk}(\bx+\bk)\rd \bx- I_{n_\bk}(f_{\bk}(\cdot+\bk)g_{\bk}(\cdot+\bk))\bigg|
		\\
		& \ll  n_{\bk}^{-a} (\log n_{\bk})^b e^{-\delta |\bk|^2} \|f\|_{\Wpgamma}
		\\
		& \ll ( n e^{-\frac{\delta}{2a}|\bk|^2} )^{-a} (\log n)^b
		e^{-\delta |\bk|^2}\|f\|_{\Wpgamma}
		\\
		&=\,    n^{-a} (\log n)^be^{-\frac{|\bk|^2\delta}{2}}\|f\|_{\Wpgamma}.
	\end{align*}
	Hence,
	\begin{align*}
		\sum_{|\bk|< \xi_n} \bigg|\int_{\IId_\bk} f_{\bk}(\bx)g_{\bk}(\bx)\rd \bx 
		 - I_{n_\bk}(f_{\bk}(\cdot+\bk)g_{\bk}(\cdot+\bk))\bigg| 
		& \ll \sum_{|\bk|< \xi_n} n^{-a}  (\log n)^be^{-\frac{|\bk|^2\delta}{2}}\|f\|_{\Wpgamma}
		\\&
		\ll n^{-a}  (\log n)^b\|f\|_{\Wpgamma}.
	\end{align*}
	For each term in the second sum  we get by H\"older's inequality and \eqref{[tau]<e^{-delta k}1},
	\begin{align*}
		\bigg|\int_{\IId_\bk} f_\bk(\bx) g_{\bk}(\bx)\rd \bx \bigg| 
		&\leq \bigg(\int_{\IId_\bk}|f_\bk(\bx)|^p g_\bk(\bx)\rd \bx \bigg)^{\frac{1}{p}} \bigg(\int_{\IId_\bk}g_\bk(\bx) \rd \bx \bigg)^{1-\frac{1}{p}} 
		\\
		&\ll e^{-\frac{|\bk - (\sign \bk)/2 |^2}{2}(1-\frac{1}{p})}\|f\|_{\Wpgamma}
		\\
		&
		\ll e^{-\delta |\bk|^2}\|f\|_{\Wpgamma},
	\end{align*}
	which implies
\begin{equation}\label{eq-epsilon01}
	\begin{aligned}
		\sum_{|\bk|\geq \xi_n}\bigg|\int_{\IId_\bk}f_\bk(\bx)g_\bk(\bx)\rd \bx\bigg|
		&
		 \ll   \sum_{|\bk|\geq \xi_n}e^{-\delta |\bk|^2} \|f\|_{\Wpgamma}
	\\
	&  \leq 2^d\sum_{s=\lceil\xi_n\rceil}^\infty   e^{-s^2 \delta}\binom{s+d-1}{d-1}\|f\|_{\Wpgamma}
	\\
	&
	\leq 2^d e^{-\xi_n^2\delta(1-\varepsilon)}\sum_{s=\lceil\xi_n\rceil}^\infty   e^{-s^2 \varepsilon \delta}\binom{s+d-1}{d-1}\|f\|_{\Wpgamma}
		\\
	&
\ll e^{-\xi_n^2\delta(1-\varepsilon)}\sum_{s=0}^\infty   e^{-s \varepsilon \delta}\binom{s+d-1}{d-1}\|f\|_{\Wpgamma}
	\end{aligned}
\end{equation}
with $\varepsilon \in (0,1/2)$. 
Using \eqref{eq-auxilary-01} we get
\begin{equation}\label{eq-epsilon02}
	\begin{aligned}
		\sum_{|\bk|\geq \xi_n}\bigg|\int_{\IId_\bk}f_\bk(\bx)g_\bk(\bx)\rd \bx\bigg| 	&  \ll  e^{-2a(1-\varepsilon)\log n}\|f\|_{\Wpgamma}		
		\ll n^{-a}  (\log n)^b\|f\|_{\Wpgamma}.
\end{aligned}
\end{equation}
	Summing up,  we have proven \eqref{IntError0}. 
	\hfill
\end{proof}

Some important quadratures such as the Frolov  quadrature and the QMC quadrature based on Fibonacci lattice rules ($d=2$) are constructively designed for functions  on $\RRd$ with support contained in  the unit  $d$-cube or for $1$-periodic functions. To employ them for  constructing  a quadrature for functions on $\RRd$ we need to modify   those constructions.

Assume that  there is a quadrature $I_m$ of the form \eqref{I_m(f)} with the integration nodes
$\{\bx_1,\ldots,\bx_m\}\subset \brac{- \frac{1}{2}, \frac{1}{2}}^d$ and weights $(\lambda_1,\ldots,\lambda_m)$ such that  the convergence rate
\begin{equation}\label{IntError-a,b,F}
	\bigg|\int_{\IId} f(\bx) \rd \bx  - I_m(f)\bigg| \leq C m^{-a} (\log m)^b \|f\|_{\Wpmix}, 
	\ \  f\in \mathring{W}^\alpha_p(\IId)
\end{equation}
holds, where $ \mathring{W}^\alpha_p(\IId)$ denotes the space of functions in  $W^\alpha_p(\RRd)$ with support contained in $\IId$.
Then based on the quadrature $I_m$, we propose two constructions of  quadratures  which approximate the integral $\int_{\RRd}f(\bx) \gamma(\rd\bx)$ with the same convergence rate for 
$f \in \Wpgamma$.
   
The first method is a preliminary change of variables  to transform the quadrature $I_m$  into a quadrature for functions in $\Wpmix$  which gives the same convergence rate, and then apply the  construction as in \eqref{I_n^gamma2}. Let us describe it.
Let $k\in \NN$ and $\psi_k$  be the function defined by
\begin{equation}\label{psin}
	\psi_k(t) = \left\{\begin{array}{rcl}
		C_k\int_{0}^t  (\frac{1}{4}-\xi^2)^k\,\rd\xi, & t\in [-\frac{1}{2},\frac{1}{2}],\\[1ex]
		\frac{1}{2},& t>\frac{1}{2},\\[1ex]
		-\frac{1}{2} ,& t<-\frac{1}{2}\,,
	\end{array}\right.
\end{equation}
where $C_k=\big(\int_{-1/2}^{1/2} (\frac{1}{4}-\xi^2)^k \,\rd\xi\big)^{-1}$. Observe that $\psi_k$ is a one-to-one mapping on $[-\frac{1}{2},\frac{1}{2}]$ and $\psi_k'$ has compact support on $[-\frac{1}{2},\frac{1}{2}]$.
If $f \in \Wpmix$,  a change of variable yields that 
	$$
\int_{\IId} f(\bx) \rd \bx =\int_{\IId}  \big(T_{\psi_k} f\big)(\bx) \rd \bx,
 $$
 where
 $$
\big( T_{\psi_k} f\big)(\bx):=\psi_k'(x_1)\cdot\ldots\cdot \psi_k'(x_d) f\big(\psi_k(x_1),\ldots,\psi_k(x_d)\big),
\ \bx \in \IId. 
 $$
Observe that the function $T_{\psi_k} f$ has support contained in $\IId$. If  $T_{\psi_k} f$ belongs to $ \mathring{W}^\alpha_p(\IId)$,  then a quadrature with the integration nodes  $\{\tilde{\bx}_1,\ldots,\tilde{\bx}_m\}\subset \IId$ and weights $(\tilde{\lambda}_1,\ldots, \tilde{\lambda}_m)$ for the function $f$ can be defined  as
$$
\tilde{I}_m(f):= I_m(T_{\psi_k} f) = \sum_{j=1}^m \tilde{\lambda}_j f(\tilde{\bx}_j), 
$$ 
where $\tilde{\bx}_j=(\psi_k(x_{j,1}),\ldots,\psi_k(x_{j,d}))$ and  $\tilde{\lambda}_j=\lambda_j \psi_k'(x_{j,1})\cdot\ldots\cdot\psi_k'(x_{j,d})$.  Hence, our task is finding a condition on $k$ so that  the mapping
\begin{equation*}
\ f \mapsto  	T_{\psi_k} f 
\end{equation*}
is a bounded operator from $\Wpmix$ to $ \mathring{W}^\alpha_p(\IId)$. A first result was proved by Bykovskii \cite{By85} where he showed that $T_{\psi_k}$ is bounded in  $W^\alpha_2(\IId)$ if 
$k\geq 2\alpha +1$. 
This result has been extended by Temlyakov, see \cite[Theorem IV.4.1]{Tem93B}, to   $W^\alpha_p(\IId)$ under the condition $
	k\geq \big\lfloor \frac{\alpha p}{p-1}\big \rfloor+1\,.
$ A recent improvement $k>\alpha
+1$ was  obtained in \cite{NUU17}. 

The second method is to decompose  functions in $W^\alpha_p(\RRd,\gamma)$ into a sum of functions on $\RR^d$ having support contained in integer translations of the $d$-cube $\IId_\theta := \big[-\frac{\theta}{2}, \frac{\theta}{2}\big]$
for a fixed   $\theta>1 $. Then the quadrature for $W^\alpha_p(\RRd,\gamma)$ is the sum of integer-translated dilations of $I_m$. Details of this construction are presented below.

 First observe that 
\begin{align*} 
	\RRd = \bigcup_{\bk \in \ZZd}	\IId_{\theta,\bk},
\end{align*}	
where $\IId_{\theta,\bk}:= \IId_\theta + \bk$. It is  well-known 
that one can constructively define a partition of unity $\brab{\varphi_\bk}_{\bk \in \ZZd}$ such that
\begin{itemize}
	\item[\rm{(i)}] $\varphi_\bk \in C^\infty_0(\RRd)$ and 
	$0 \le \varphi_\bk (\bx)\le 1$, \ \ $\bx \in \RRd$, \ \ $\bk \in \ZZd$;
	\item[\rm{(ii)}]  $\supp \varphi_\bk$ are contained in the interior of  $\IId_{\theta,\bk}$, $\bk \in \ZZd$;
	\item[\rm{(iii)}]  $\sum_{\bk \in \ZZd}\varphi_\bk (\bx)= 1$, \ \ $\bx \in \RRd$;
\item[\rm{(iv)}]  $\norm{\varphi_\bk }{W^\alpha_p(\IId_{\theta,\bk})} \le C_{\alpha,d,\theta}$, \ \ 
$\bk \in \ZZd$,
\end{itemize}
(see, e.g., \cite[Chapter VI, 1.3]{Stein1970}).  By the items (ii) and (iii)
the integral 
	$\int_{\RRd}f(\bx) \gamma(\rd\bx)$ can be represented as
\begin{align} \label{Int_RRd-F}
	\int_{\RRd}f(\bx) \gamma(\rd\bx) 
	=  \sum_{\bk \in \ZZd}\int_{\IId_{\theta,\bk}} f_{\theta,\bk}(\bx) g_{\theta,\bk}(\bx) \varphi_\bk (\bx) \rd \bx,
\end{align}
where  $f_{\theta,\bk}$ and $g_{\theta,\bk}$ denote the restrictions of $f$ and $g$ on $\IId_{\theta,\bk}$, respectively.
The quadrature \eqref{I_m(f)} induces the quadrature 
\begin{equation}\label{I_theta,m(f)}
	I_{\theta,m}(f): = \sum_{i=1}^m \lambda_{\theta,i} f(\bx_{\theta,i}), 
\end{equation}
for functions $f$ on $\IId_\theta$,  where 
$\bx_{\theta,i}:= \theta\bx_i$ and 
$\lambda_{\theta,i}:= \theta \lambda_i $.

Denote by $ \mathring{W}^\alpha_p(\IId_\theta)$ the subspace of functions in $W^\alpha_p(\RRd)$ with support contained in $\IId_\theta$.
From  \eqref{IntError-a,b,F} the error bound
\begin{equation*}\label{IntError-a,b,theta}
	\bigg|\int_{\IId_\theta} f(\bx) \rd \bx  - I_{\theta,m}(f)\bigg| \ll m^{-a} (\log m)^b \|f\|_{W^\alpha_p(\IId_\theta)}
\end{equation*}
holds for every $f\in \mathring{W}^\alpha_p(\IId_\theta)$.
Let $f\in \Wpgamma$. It is clear that 
$f_{\theta,\bk}(\cdot+\bk)\in W^\alpha_p(\IId_\theta)$ and similar to \eqref{eq:b4} and \eqref{eq:norm-fwid} we get
\begin{equation*}\label{ineq-norms}
	\|f_{\theta,\bk}(\cdot+\bk)\|_{W^\alpha_p(\IId_\theta)} 
	\ll e^{\frac{|\bk + (\theta \sign \bk)/2 |^2}{2p}}\|f\|_{\Wpgamma}, \ \
	f\in \Wpgamma, \ \ \bk \in \ZZd. 
\end{equation*}
Similarly to \eqref{multipl-algebra1} and \eqref{multipl-algebra2}, by additionally using the items (ii)  and (iv)  we have that 
\begin{equation*} 
	f_{\theta,\bk}(\cdot+\bk)g_{\theta,\bk}(\cdot+\bk)\varphi_\bk (\cdot+\bk)\in \mathring{W}^\alpha_p(\IId_\theta),
\end{equation*}
and 
\begin{align*}
	\|f_{\theta,\bk}(\cdot+\bk)g_{\theta,\bk}(\cdot+\bk)\varphi_\bk(\cdot+\bk)\|_{W^\alpha_p(\IId_\theta)} 
	\ll 
e^{\frac{|\bk + (\theta \sign \bk)/2 |^2}{2p}-\frac{|\bk|^2}{2\tau}} \|f\|_{\Wpgamma},
\end{align*}
where $\tau$ is a fixed number satisfying the inequalities $1<\tau<p<\infty $.
We choose $\delta>0$ so that
\begin{equation*}  \label{[tau]<e^{-delta k}}
\max \bigg\{e^{-\frac{|\bk - (\theta \sign \bk)/2 |^2}{2}(1-\frac{1}{p})}, 
	 e^{\frac{|\bk + (\theta \sign \bk)/2 |^2}{2p}-\frac{|\bk|^2}{2\tau}}\bigg\}
	\leq 
	C e^{-\delta |\bk|^2}, \ \  \bk\in \ZZd.
\end{equation*}

For $n\in \RR_1$, let $\xi_n$ and $n_{\bk}$ be given as in \eqref{xi-int}	
and \eqref{n_bk}, respectively. 
Noting \eqref{Int_RRd-F} and \eqref{I_theta,m(f)}, we define
\begin{equation*} 
I_{\theta,n}(f):=\sum_{|\bk|< \xi_n}I_{\theta,n_\bk}
	\big(f_{\theta,\bk}(\cdot+\bk)g_{\theta,\bk}(\cdot+\bk)\varphi_\bk(\cdot+\bk)\big),
\end{equation*}
or equivalently,
\begin{equation} \label{I_n^gamma3}
I_{\theta,n}(f):=
		\sum_{|\bk|< \xi_n}\sum_{j=1}^{\lfloor n_\bk \rfloor} \lambda_{\theta,\bk,j} f(\bx_{\theta,\bk,j})
\end{equation}
as a linear quadrature for the approximate integration of $\gamma$-weighted functions $f$ on $\RRd$
where $\bx_{\theta,\bk,j}:= \bx_{\theta,j}+\bk$ and 
$\lambda_{\theta, \bk,j}:= \lambda_{\theta,j} g_{\bk}(\bx_{\theta,\bk,j})\varphi_\bk(\bx_{\theta,\bk,j})$.   The integration nodes of the quadrature  $I_{\theta,n}$ are
\begin{equation} \label{int-nodes-theta}
\{\bx_{\theta,\bk,j}: |\bk|< \xi_n, \, j=1,\ldots,\lfloor n_\bk \rfloor\}
\subset \RRd,
\end{equation}
and the  weights 
$$
(\lambda_{\theta,\bk,j}: |\bk|< \xi_n, \, j=1,\ldots,\lfloor n_\bk \rfloor).
$$
Due to \eqref{<n2}, the number of integration nodes  is not greater than $n$. Moreover, from the definition we can see that the integration nodes are contained in the ball of radius $\xi_{\theta,n}^*:= \theta\sqrt{d}/2 + \xi_n$, i.e., 
$$
\{\bx_{\theta,\bk,j}: |\bk|< \xi_{\theta,n}^*, \, j=1,\ldots,\lfloor n_\bk \rfloor\}
\subset B(\xi_{\theta,n}^*):= \brab{\bx \in \RRd:\, |\bx| \le \xi_{\theta,n}^*}.$$
Notice  that the set of integration nodes \eqref{int-nodes-theta} possesses  similar sparsity properties as  the set \eqref{int-nodes}.

In a way similar to the proof of Theorem \ref{thm:int-general} we derive

\begin{theorem} \label{thm:int-general2}
	Let $\alpha\in \NN$, $1<p<\infty$ and $a >0$, $b \ge 0$, $ \theta>1$. 
	Assume that for any $m \in \RR_1$, there is  a   quadrature  $I_m$  of the form \eqref{I_m(f)} with 
	$\{\bx_1,\ldots,\bx_m\}\subset \brac{- \frac{1}{2}, \frac{1}{2}}^d$ satisfying \eqref{IntError-a,b,F}. Then for  the   quadrature  $I_{\theta,n}$ defined as in \eqref{I_n^gamma3} we have 
	\begin{equation} 	\label{IntError}
		\bigg|\int_{\RRd}f(\bx) \gamma(\rd\bx) - 	I_{\theta,n}(f)\bigg| 
		\ll 
		n^{-a}  (\log n)^b \|f\|_{\Wpgamma}, 
		\ \  f \in \Wpgamma.
	\end{equation}
\end{theorem}

As noticed in Introduction, we do not study the dimension dependence for error estimates of integration. Hence the hidden constant in the bound \eqref{IntError} may depend  on the dimension $d$ and may increase exponentially in $d$. Therefore, for very large $d$, the resulting algorithm may not be practical.

\subsection{ Asymptotic order of optimal numerical integration}\label{sec-optimal}
In this subsection, we prove the asymptotic order of optimal numerical integration as formulated in \eqref{AsympQuadrature} based on Theorem \ref{thm:int-general2} and known results on  numerical integration for functions from $\Wpmix$.

\begin{theorem} \label{thm:main}
		Let $\alpha\in \NN$ and $1<p<\infty$.  Then one can construct  an asymptotically optimal  family of quadratures  of the form  \eqref{I_n^gamma3} $\big(I_n^\gamma\big)_{n \in \RR_1}$ such that
	\begin{equation}\label{eq:InWg}
		\sup_{f\in \BWpgamma} \bigg|\int_{\RRd}f(\bx) \gamma(\rd\bx) - I_n^\gamma(f)\bigg| 
	\asymp
	\Int_n\big(\BWpgamma\big) 
	\asymp 
	n^{-\alpha} (\log n)^{\frac{d-1}{2}}.
	\end{equation}
\end{theorem}

\begin{proof} 
Let $I_{{\rm F},m}$ be the  Frolov quadrature for functions in $\mathring{W}^\alpha_p(\IId)$ (see, e.g.,  \cite[Chapter 8]{DTU18B}  for the definition)    in
the form \eqref{I_m(f)} with 
$\{\bx_1,\ldots,\bx_m\}\subset \brac{- \frac{1}{2}, \frac{1}{2}}^d$. 
 It was proven   in \cite{Florov1976} for $p=2$, and in \cite{Skriganov1994} for $1 < p < \infty$  that
 \begin{equation}\label{OptIntegration}
 	\bigg|\int_{\IId} f(\bx) \rd \bx  - I_{{\rm F},m}(f)\bigg| \leq C m^{-\alpha} (\log m)^{\frac{d-1}{2}} \|f\|_{\mathring{W}^\alpha_p(\IId)}, 
 	\ \  f\in \mathring{W}^\alpha_p(\IId).
 \end{equation}		
For a fixed $ \theta>1$, we define $I_n^\gamma:=  I_{\theta,n}$  as the quadrature described in  Theorem \ref{thm:int-general2} for $a= \alpha$ and $b = {\frac{d-1}{2}}$, based on  $I_m = {I_{{\rm F},m}}$. By 
Theorem \ref{thm:int-general2} and \eqref{OptIntegration} we prove the upper bound in \eqref{eq:InWg}.
	
	Since for $f\in \mathring{W}^\alpha_p(\IId)$
	$$
	\|f\|_{\Wpgamma} \leq (2\pi)^{-\frac{d}{2p}} \|f\|_{\mathring{W}^\alpha_p(\IId)},
	$$
	we get
	\begin{align*}  
		\Int_n\big(\BWpgamma\big) \gg \Int_n(\mathring{\boldsymbol{W}}^\alpha_p(\IId)).
	\end{align*}	
Hence the lower bound in \eqref{eq:InWg} follows from the lower bound 
$\Int_n(\mathring{\boldsymbol{W}}^\alpha_p(\IId)) \gg  n^{-\alpha} (\log n)^{\frac{d-1}{2}}$ proven in \cite{Tem1990}. 
%
	\hfill
\end{proof}

Besides Frolov quadratures, there are many  quadratures for efficient numerical integration for functions on $\IId$ to list. We refer the reader to \cite[Chapter 8]{DTU18B} for  bibliography  and historical comments as well as related results, in particular,
the asymptotic order
\begin{equation*}\label{IntW(IId)}
\Int_m\big(\BWpmix\big) \asymp m^{-\alpha} (\log m)^{\frac{d-1}{2}}
\end{equation*}
for $1 < p < \infty$.
We recall only some of them, especially those which give asymptotic order of optimal integration.

A quasi-Monte Carlo (QMC) quadrature based on a set of integration nodes $\{\bx_1,\ldots,\bx_m\}\subset \IId$  is of the form
\begin{equation*}\label{I_m(f)-QMC}
I_m(f) = \frac{1}{m}\sum_{i=1}^m f(\bx_i). 
\end{equation*}
In  \cite{Dick2007, Dick2008} for a prime number $q$  the author introduced higher order digital nets over the finite field 
$\FF_q:=\brab{0, 1, \ldots, q-1}$ equipped with the arithmetic operations modulo $q$.  Such digital nets can achieve the convergence rate $m^{-\alpha}(\log m)^{d \alpha}$ with $m = q^s$  for functions from $W^\alpha_2(\IId)$, see \cite{DP2010}.
In the recent paper \cite{GSY2018},  the authors have shown that  the asymptotic order  of
$\Int_m\big(\bW^\alpha_2(\IId)\big)$  can be achieved  by Dick's digital nets $\{\bx_1^*,\ldots,\bx_{q^s}^*\}$ of order $(2\alpha + 1)$. Namely, they proved that 
\begin{equation}\label{OptQMC}
	\bigg|\int_{\IId} f(\bx) \rd \bx  -  \frac{1}{m} \sum_{i=1}^m f(\bx_i^*)\bigg| \leq C 
	m^{-\alpha} (\log m)^{\frac{d-1}{2}} \|f\|_{W^\alpha_2(\IId)}, 
	\ \  f\in \ W^\alpha_2(\IId), \ \ m=q^s.
\end{equation}		

In the case $d=2$ the QMC quadrature $I_m=I_{\Phi,m}$ based on Fibonacci lattice rules ($d=2$)  is also  asymptotically optimal for numerical integration of  periodic functions in   $\tilde{W}^\alpha_p(\II^2)$, that is,
\begin{equation}\label{OptInt-Fibonacci}
	\bigg|\int_{\II^2} f(\bx) \rd \bx  - I_{\Phi,m}(f)\bigg| \leq C m^{-\alpha} (\log m)^{\frac{1}{2}} \|f\|_{W^\alpha_p(\II^2)}, 
	\ \  f\in \tilde{W}^\alpha_p(\II^2),
\end{equation}
where $\tilde{W}^\alpha_p(\II^2)$ denotes the subspace of $W^\alpha_p(\II^2)$ of all functions which can be extended to the whole $\RR^2$ as $1$-periodic  functions in each variable. The estimate \eqref{OptInt-Fibonacci} was proven in \cite{Bakhvalov1963} for $p=2$ and in \cite{Tem1991}
 for $1 < p < \infty$.
The QMC quadrature $I_m=I_{\Phi,m}$ based on Fibonacci lattice rules ($d=2$)  is defined by
\begin{equation*}\label{Phi_m(f)}
I_{\Phi,m}(f): = \frac{1}{b_m} \sum_{i=1}^{b_m} f\bigg(\Big\{\frac{i}{b_m}\Big\} - \frac{1}{2}, \Big\{\frac{ib_{m-1}}{b_m}\Big\} - \frac{1}{2}\bigg),
\end{equation*}
where $b_0 = b_1 = 1$,  $b_m := b_{m-1} + b_{m-2}$ are the Fibonacci numbers and $\brab{x}$ denotes the fractional part of the number $x$.

Therefore,  from Theorems \ref{thm:int-general}--\ref{thm:main} and  \eqref{OptQMC}, \eqref{OptInt-Fibonacci} it follows that the QMC quadratures based on Dick's digital nets  of order $(2\alpha + 1)$ and Fibonacci lattice rules ($d=2$)  can be used for  assembling  asymptotically optimal quadratures  $I_n^\gamma$ and $I_{\theta,n}^\gamma$  of the forms \eqref{I_n^gamma2} and \eqref{I_n^gamma3} for 
$\Int_n\big(\BWpgamma\big)$, in the particular cases $p=2$, $d \ge 2$, and $1 < p < \infty$, $d=2$, respectively.

The sparse Smolyak grid $SG(\xi)$ in $\IId$ is defined as  the set of points:
\begin{equation*}\label{Smolyak-net}
	SG(\xi):= \brab{\bx_{\bk,\bs}:= 2^{-\bk}\bs \in \ZZd: \, |\bk|_1 \le \xi, \ \ | s_i| \le 2^{k_i - 1}, \ i=1,\ldots,d}, 
	\ \ \xi  \in \RR_1. 
\end{equation*}
For a given $m \in \RR_1$, let $\xi_m$ be the maximal number satisfying $|SG(\xi_m)| \le m$. Then we can constructively define a quadrature $I_m = {I_{{\rm S},m}}$ based on the integration nodes in $SG(\xi_m)$ so that
\begin{equation}\label{OptInt-Smolyak}
	\bigg|\int_{\IId} f(\bx) \rd \bx  - {I_{{\rm S},m}}(f)\bigg| \leq C m^{-\alpha} (\log m)^{(d-1)(\alpha + 1/2)} \|f\|_{\Wpmix}, 
	\ \  f\in W^\alpha_p(\IId).
\end{equation}
To understand this quadrature let us recall  a detailed construction from \cite[page 760]{DU2015}.
Indeed, from the well-known embedding of $\Wpmix$ into the Besov space  of mixed smoothness $B^\alpha_{p,\max(p,2)}(\IId)$  (see, e.g., \cite[Lemma 3.4.1(iv)]{DTU18B}), and the result on B-spline sampling  recovery of functions from the last space   it follows  that one can constructively define a sampling recovery algorithm of the form 
\begin{equation*}\label{R_m(f)}
	R_m(f): = \sum_{\bx_{\bk,\bs}\in SG(\xi_m)} f(\bx_{\bk,\bs}) \phi_{\bk,\bs}
\end{equation*}
with certain B-splines $\phi_{\bk,\bs}$, such that
\begin{equation*}\label{Sampling-Smolyak}
	\norm{f - R_m(f)}{L_1(\IId)}	 \leq C m^{-\alpha} (\log m)^{(d-1)(\alpha + 1/2)} \|f\|_{\Wpmix}, 
	\ \  f\in W^\alpha_p(\IId).
\end{equation*}
Then the quadrature ${I_{{\rm S},m}}$ can be defined as
\begin{equation*}\label{S_m(f)}
{I_{{\rm S},m}}(f): = \sum_{\bx_{\bk,\bs} \in SG(\xi_m)} \lambda_{\bk,\bs} f(\bx_{\bk,\bs}) ,  \ \ 
	\lambda_{\bk,\bs}:= \int_{\IId} \phi_{\bk,\bs}(\bx) \rd \bx,
\end{equation*} 
and \eqref{OptInt-Smolyak} is implied by the obvious inequality 
$\big|\int_{\IId} f(\bx) \rd \bx  - I_{{\rm S},m}(f)\big| \le \norm{f - R_m(f)}{L_1(\IId)}$.
Therefore,  from Theorem \ref{thm:int-general} and  \eqref{OptInt-Smolyak} we can see that the  Smolyak quadrature ${I_{{\rm S},m}}$  can be used for  assembling  a  quadrature  
	${I_{{\rm S},n}}$ of the form \eqref{I_n^gamma2} with ``double" sparse integration nodes which gives the convergence rate
	\begin{equation*}
\bigg|\int_{\RRd}f(\bx) \gamma(\rd\bx) - {I_{{\rm S},n}}(f)\bigg| 
	\ll
	n^{-\alpha}  (\log n)^{(d-1)(\alpha + 1/2)}, \ \ 	f \in \BWpgamma.
\end{equation*}

\section{Approximation}
\label{Approximation}
In this section we study the 
linear  approximation and sampling recovery in $L_q(\RRd,\gamma)$ of  functions from $\Wpgamma$, and the asymptotic optimality in terms 
of Kolmogorov $n$-widths and the linear $n$-widths and sampling $n$-widths for $1\le q < p <\infty$ and $p=q=2$. 

 Let $n \in \NN$ and 
let $X$ be a Banach space and $F$ a central symmetric compact set in $X$.
Then the Kolmogorov $n$-width  of $F$ is defined by
\begin{equation*}
d_n(F,X):= \inf_{L_{n}}\sup_{f\in F}\inf_{g\in L_n}\|f-g\|_X, 
\end{equation*}
where the left-most  infimum is  taken over all  subspaces $L_{n}$ of dimension  $\le n$ in $X$.  
The linear $n$-width of the set $F$ is defined by
$$
\lambda_n(F,X):=\inf_{A_n} \sup_{f\in F} \|f-A_n(f)\|_X,
$$
where the infimum is taken over all linear operators $A_n$ in  $X$ with ${\rm rank}\, A_n\leq n$. Notice that  if $X$ is a Hilbert space, then $\lambda_n(F,X) = d_n(F,X)$.

Let  $\Omega$ be a domain in  $\RRd$.  Let $n \in \NN$ and let
 $X$ be a Banach space of functions on $\Omega$ and $F$ a compact set in $X$.
Given $\{\bx_i\}_{i=1}^n\subset \Omega$, to approximately recover $f \in F$  from the sampled values $\brab{f(\bx_i)}_{i=1}^n$ we use a (linear) sampling algorithm defined by
\begin{equation} \label{R_n(f)}
	R_n(f): = \sum_{i=1}^n  f(\bx_i) \varphi_i,
\end{equation}
where $\brab{\varphi_i}_{i=1}^n$ is a collection of $n$ functions in $X$. For convenience, we assume that some points from $\{\bx_i\}_{i=1}^n\subset \Omega$ and some functions from $\brab{\varphi_i}_{i=1}^n$ may coincide.
For $n\in \NN$ we define the sampling $n$-width  of the set $F$ in $X$ as
$$
\varrho_n(F,X):=\inf_{\bx_1,\ldots,\bx_n\in \Omega,\atop
	\varphi_1,\ldots,\varphi_n\in X} \ \sup_{f\in F}
\|f- R_n(f)\|_X,
$$
where $R_n(f)$ is given by \eqref{R_n(f)}. Obviously, we have the inequalities
\begin{equation}\label{eq-relations}
d_n(F,X) \leq \lambda_n(F,X)\leq \varrho_n(F,X).
\end{equation}

There are other popular $n$-widths in approximation theory like the entropy $n$-widths, Gel'fand $n$-widths and Bernstein $n$-widths, etc. In particular, for optimality of numerical algorithms, the Gel'fand $n$-widths are very important, since optimal algorithms could be non-linear (for detail, see, e.g., \cite[Section 6 and Section 9.6]{DTU18B}). However, these $n$-widths are not in the scope of consideration of the present paper.

For technical  convenience we use the conventions $A_n := A_{\lfloor n \rfloor}$, $R_n := R_{\lfloor n \rfloor}$, 
$d_n(F,X) := d_{\lfloor n \rfloor}(F,X)$, $\lambda_n(F,X) := \lambda_{\lfloor n \rfloor}(F,X)$ and  $\varrho_n(F,X) := \varrho_{\lfloor n \rfloor}(F,X)$ for $n \in \RR_1$.

For given $\alpha$ and $p,q$, we  make use of the abbreviations:
$$
\lambda_n := \lambda_n(\BWpgamma,\Lqgamma), \quad d_n := d_n(\BWpgamma,\Lqgamma) , 
$$
$$
\varrho_n := \varrho_n(\BWpgamma,\Lqgamma).
$$
We prove the asymptotic orders  of $\lambda_n$, $d_n$ and $\varrho_n$ as well as constructively define asymptotically optimal linear approximation methods which are very different for  the cases $1\le q < p <\infty$ and $q = p =2$.

\subsection{The case $1\le q < p < \infty$}
Let $\alpha\in \NN$, $1 \le q < p < \infty$ and $a >0$, $b \ge 0$. Denote by $\tilde{L}_q(\IId)$ and $\tilde{W}^\alpha_p(\IId)$ the subspaces of  $L_q(\IId) $ and $\Wpmix$, respectively,  of all functions $f$ which can be extended to the whole $\RRd$ as $1$-periodic  functions in each variable (denoted again by $f$).
 Let $A_m$  be a linear operator  in $\tilde{L}_q(\IId)$ of rank $\leq m$. Assume it holds that
\begin{equation}\label{A_m-Error-a,b-theta}
	\| f - A_m(f) \|_{\tilde{L}_q(\IId)}\leq C m^{-a} (\log m)^b \|f\|_{\tilde{W}^\alpha_p(\IId)}, 
	\ \  f\in \tilde{W}^\alpha_p(\IId).
\end{equation}
Then based on  $A_m$, we will construct  a linear operator  
$A_m^\gamma$ in $L_q(\RRd, \gamma)$ which approximates  $f \in \Wpgamma$  with the same convergence rate. Our strategy is similar to the problem of numerical  integration considered  in 
Subsection \ref{Quadratures on sparse-digital-nets }.  

Fix a number $\theta$ with   $\theta>1$.  
Denote by $\tilde{L}_q(\IId_\theta)$ and $\tilde{W}^\alpha_p(\IId_\theta)$ the subspaces of  $L_q(\IId_\theta) $ and $W^\alpha_p(\IId_\theta)$, respectively,  of all functions $f$ which can be extended to the whole $\RRd$ as $\theta$-periodic  functions in each variable (denoted again by $f$).
A linear operator $A_m$  induces the linear operator $A_{\theta,m}$ in $\tilde{L}_q(\IId_\theta)$, defined 
for $f \in \tilde{L}_q(\IId_\theta)$ by $A_{\theta,m}(f):= A_m(f(\cdot/\theta))$.

From  \eqref{A_m-Error-a,b-theta} it follows that
\begin{equation*}\label{A_m-Error-a,b-theta2}
	\| f - A_{\theta,m}(f) \|_{\tilde{L}_q(\IId_\theta)}\leq C m^{-a} (\log m)^b \|f\|_{\tilde{W}^\alpha_p(\IId_\theta)}, 
	\ \  f\in \tilde{W}^\alpha_p(\IId_\theta).
\end{equation*}

Since $q<p$,  we can choose a fixed $\delta>0$ such that
\begin{equation}  \label{<e^{-delta k}}
		 {e^{\frac{|\bk + (\theta \sign \bk)/2 |^2}{2p} 
		 	-\frac{|\bk - (\theta\sign \bk)/2 |^2}{2q}}}
		 \leq 
		 C e^{- \delta |\bk|^2}, \ \  \bk\in \ZZd.
\end{equation}
For $n\in \RR_1$, let $\xi_n$ and $n_\bk$ be given as in \eqref{xi-int} and \eqref{n_bk}.
Recall that  we write
$\IId_{\theta,\bk}:=\bk+\IId_\theta$  for $\bk \in \ZZd$, and
$ f_{\theta,\bk} $ the restriction of $f$ on $\IId_{\theta,\bk}$  for a function $f$ on $\RRd$.  Let  $\brab{\varphi_\bk}_{\bk \in \ZZd}$ be the partition of unity satisfying items (i)--(iv),  introduced in Subsection \ref{Quadratures on sparse-digital-nets }.
Similarly to \eqref{multipl-algebra1} and \eqref{multipl-algebra2}, by additionally using the items (ii)  and (iv)  we have that if $f \in \Wpgamma$, then
\begin{equation*}\label{f_theta,bk}
	f_{\theta,\bk}(\cdot+\bk)\varphi_\bk (\cdot+\bk)\in \tilde{W}^\alpha_p(\IId_\theta),
\end{equation*}
and it holds that
\begin{equation}  \label{eq:norm-fwid2}
	\|f_{\theta,\bk}(\cdot+\bk)\varphi_\bk(\cdot+\bk)\|_{\tilde{W}^\alpha_p(\IId_\theta)} 
	\ll 
 {e^{\frac{|\bk + (\theta \sign \bk)/2 |^2}{2p}}}\|f\|_{\Wpgamma}.
\end{equation}
We define the linear operator $	A_{\theta,n}^\gamma$ in $L_q(\RRd,\gamma)$ of rank $\le n$ by
\begin{equation}  \label{A_n^gamma}
	\brac{A_{\theta,n}^\gamma f}(\bx): = 
	\sum_{|\bk|< \xi_n} 
	\brac{A_{\theta, n_\bk}\tilde{f}_{\theta,\bk}}(\bx-\bk),
\end{equation} 
where $\tilde{f}_{\theta,\bk}(\bx)=f_{\theta,\bk}(\bx+\bk)\varphi_\bk (\bx+\bk)$.
Indeed, by \eqref{<n2},
\begin{align*} 
	{\rm rank}\, A_{\theta,n}^\gamma
	\le 
	\sum_{|\bk|< \xi_n}  {\rm rank}\, A_{\theta, n_\bk} 
	\le	\sum_{|\bk|< \xi_n}n_\bk  \le n.
\end{align*}

\begin{theorem} \label{thm:approx-general-theta}
	Let $\alpha\in \NN$, $1\le q < p <\infty$ and $a >0$, $b \ge 0$, $  \theta >1$.  
	Assume that for any $m \in \RR_1$, there is  a linear operator
$A_m$  in $\tilde{L}_q(\IId)$ of rank $\leq m$ 
	such that the convergence rate \eqref{A_m-Error-a,b-theta} holds. Then for any $n \in \RR_1$, based on this linear operator one can construct the linear operator
	$A_{\theta,n}^\gamma$ in $L_q(\RRd,\gamma)$ of rank $\leq n$ as in \eqref{A_n^gamma}  so that
	\begin{equation}\label{A_m-Error-a,b}
		\| f - A_{\theta,n}^\gamma(f) \|_{\Lqgamma}\leq C n^{-a} (\log n)^b \|f\|_{\Wpgamma}, 
		\ \  f\in \Wpgamma.
	\end{equation}
\end{theorem}

\begin{proof} 
	The proof of this theorem is analogous to that of Theorem \ref{thm:int-general2} with certain modifications. We give a short description of it.
	From the items (ii) and (iii) in Subsection \ref{Quadratures on sparse-digital-nets } it is implied that
	\begin{align*}  
		f
		=  \sum_{\bk \in \ZZd}f_{\theta,\bk}\varphi_\bk. 
	\end{align*}
	Hence we have
\begin{align} \label{f-A_n'^gf}
		\|f- A_{\theta,n}^\gamma(f)\|_{\Lqgamma}
		&\leq
		\sum_{|\bk|< \xi_n} 
		\norm{f_{\theta,\bk}\varphi_\bk - \brac{A_{\theta, n_\bk}\tilde{f}_{\theta,\bk}}(\cdot -\bk)}
		{L_q(\IId_{\theta,\bk}, \gamma)}
+ \sum_{|\bk|\geq \xi_n} 	\norm{f_{\theta,\bk}\varphi_\bk}{L_q(\IId_{\theta,\bk}, \gamma)}. 
	\end{align}
From	\eqref{n_bk}, \eqref{A_m-Error-a,b-theta} and \eqref{eq:norm-fwid2} we derive the estimates
	\begin{align*}
	&	\norm{f_{\theta,\bk}\varphi_\bk - \brac{A_{\theta, n_\bk}\tilde{f}_{\theta,\bk}}(\cdot -\bk)}
		{L_q(\IId_{\theta,\bk}, \gamma)} 
		\\
		& \ll e^{- \frac{|\bk - (\theta\sign \bk)/2 |^2}{2q}}
		\norm{f_{\theta,\bk}(\cdot+\bk)\varphi_\bk(\cdot+\bk) - A_{\theta, n_\bk}\tilde{f}_{\theta,\bk}}
		{\tilde{L}_q(\IId_{\theta})} 
		\\
		& \ll  e^{- \frac{|\bk - (\theta\sign \bk)/2 |^2}{2q}}  n_{\bk}^{-a} (\log n_{\bk})^b \|f(\cdot+\bk)\varphi_{\bk}(\cdot+\bk)\|_{\tilde{W}^\alpha_p(\IId_\theta)}
		\\
		&
		\ll   e^{\frac{|\bk + (\theta \sign \bk)/2 |^2}{2p}-
				\frac{|\bk - (\theta\sign \bk)/2 |^2}{2q}}
		\Big( n e^{-\frac{\delta}{2a}|\bk|^2} \Big)^{-a} (\log n)^b \|f\|_{\Wpgamma}.
	\end{align*}
Using  \eqref{<e^{-delta k}} we get
	\begin{align*}
	\norm{f_{\theta,\bk}\varphi_\bk - \brac{A_{\theta, n_\bk}\tilde{f}_{\theta,\bk}}(\cdot -\bk)}
	{L_q(\IId_{\theta,\bk}, \gamma)} 
	&
	\ll  e^{- \frac{\delta}{2} |\bk|^2}  n^{-a}  (\log n)^b  \|f\|_{\Wpgamma},
\end{align*}
which implies
	\begin{align*}
		\sum_{|\bk|< \xi_n} \norm{f_{\theta,\bk}\varphi_\bk - \brac{A_{\theta, n_\bk}\tilde{f}_{\theta,\bk}}(\cdot -\bk)}
		{L_q(\IId_{\theta,\bk}, \gamma)} 
		&\ll \sum_{|\bk|< \xi_n}   e^{- \frac{ \delta}{2} |\bk|^2} 
		n^{-a}  (\log n)^b  \|f\|_{\Wpgamma}
		\\&
		\ll  n^{-a}  (\log n)^b  \|f\|_{\Wpgamma}.
	\end{align*}
Similar to \eqref{eq-epsilon01} and \eqref{eq-epsilon02}, we have for a fixed $\varepsilon \in (0,1/2)$,
\begin{align*}
	\sum_{|\bk|\geq \xi_n} 	\norm{f_{\theta,\bk}\varphi_\bk}{L_q(\IId_{\theta,\bk}, \gamma)}
		& \ll  \sum_{|\bk|\geq \xi_n} 
		 e^{- \frac{|\bk - (\theta \sign \bk)/2 |^2}{2q}+
				\frac{|\bk + (\theta\sign \bk)/2 |^2}{2p}}
		 \|f\|_{\Wpgamma}
	\\
	&	\ll  \sum_{|\bk|\geq \xi_n}  e^{- \delta |\bk|^2}  \|f\|_{\Wpgamma}
	 \ll e^{- \delta (1-\varepsilon) \xi_n^2}  \|f\|_{\Wpgamma} 
	\\
	&	 =\, e^{-2 a (1-\varepsilon) \log n}  \|f\|_{\Wpgamma}
		\ll  n^{-a}  (\log n)^b  \|f\|_{\Wpgamma}.
	\end{align*}
	From the last two estimates and \eqref{f-A_n'^gf} we obtain \eqref{A_m-Error-a,b}.
	\hfill
\end{proof}

%

\begin{lemma} \label{lemma: widths}
	Let $\alpha\in \NN$ and $1\le q<p<\infty$.
	Then we have
	\begin{equation*}
		d_m(\tilde{\bW}^\alpha_p(\IId),\tilde{L}_q(\IId)) \asymp m^{-\alpha} (\log m)^{(d-1)\alpha}.
	\end{equation*}
	Moreover,   truncations  on certain hyperbolic crosses of the Fourier series form an asymptotically optimal  linear operator $A_m$ in $\tilde{L}_q(\IId)$  of rank  $\le m$ such that 
	\begin{equation}\label{f-A_m f}
		\|f-A_m (f)\|_{\tilde{L}_q(\IId)} \ll m^{-\alpha} (\log m)^{(d-1)\alpha} \|f\|_{\tilde{W}^\alpha_p(\IId)}, 
		\ \  f\in \tilde{W}^\alpha_p(\IId).
	\end{equation}		
\end{lemma}
For details on this lemma see, e.g., in \cite[Theorems 4.2.5, 4.3.1 \& 4.3.7]{DTU18B} and related comments on the asymptotic optimality of the hyperbolic cross approximation.

We are now in the position to prove the main result in this section.

\begin{theorem} \label{theorem:widths: q<p}	
	Let $\alpha\in \NN$ and $1\le q<p<\infty$. Then for any $n \in \RR_1$, based on the linear operator $A_m$ in Lemma \ref{lemma: widths} one can construct the linear operator
	$A_n^\gamma$ in  $L_q(\RRd,\gamma)$ of rank $\leq n$ as in \eqref{A_n^gamma}  so that
	\begin{equation}\label{eq:dnWg}
			\sup_{f\in \BWpgamma}	\| f - A_n^\gamma(f) \|_{\Lqgamma}
		\asymp
		\lambda_n\asymp	d_n \asymp
		 n^{-\alpha} (\log n)^{(d-1)\alpha}.
	\end{equation}
Moreover, with the additional condition $q=2$,
\begin{align}\label{sampling-widths:p>2}
	\varrho_n
	\asymp 
 n^{-\alpha} (\log n)^{(d-1)\alpha}.
\end{align}		
\end{theorem}
\begin{proof}
For a fixed $\theta>1$, we define $A_n^\gamma:=  A_{\theta,n}^\gamma$  as the linear operator described in  Theorem~\ref{thm:approx-general-theta}.
	The upper bounds  in \eqref{eq:dnWg} follow from \eqref{f-A_m f} and Theorem \ref{thm:approx-general-theta} with $a = \alpha$, $b =(d-1)\alpha$.

If $f$ is a $1$-periodic function on $\RRd$ and $f\in  \tilde{W}^\alpha_p(\IId)$, then
\begin{align*}
\|f\|_{\Wpgamma}&= \Bigg((2\pi)^{-d/2}\sum_{|\br|_\infty \leq \alpha} \int_{\RRd} |D^\br f(\bx)|^p e^{-\frac{|\bx|^2}{2}}\rd \bx\Bigg)^{1/p}
\\
& = (2\pi)^{-\frac{d}{2p}}\Bigg(\sum_{|\br|_\infty \leq \alpha} \sum_{\bk\in \ZZd} \int_{\IId} |D^\br f(\bx+\bk)|^p e^{-\frac{|\bx+\bk|^2}{2}}\rd \bx\Bigg)^{1/p}
\\
& \ll \Bigg(\sum_{|\br|_\infty \leq \alpha}  \int_{\IId} |D^\br f(\bx)|^p \rd \bx \sum_{\bk\in \ZZd}e^{-\frac{|\bk - (\sign \bk)/2 |^2}{2}}\Bigg)^{1/p}
\\
&
\ll 
 \|f\|_{\tilde{W}^\alpha_p(\IId)},
\end{align*}
and
$$
\|f\|_{\tilde{L}_q(\IId)}=\bigg((2\pi)^{\frac{d}{2}}\int_{\IId}|f(\bx)|^qe^{\frac{|\bx|^2}{2}}g(\bx)\rd \bx\bigg)^{1/q} \leq (2\pi)^{\frac{d}{2q}}e^{\frac{d}{8q}}\|f\|_{\Lqgamma}.
$$
Hence we get
\begin{align*}  
\lambda_n \ge	d_n \gg	d_n(\tilde{\bW}^\alpha_p(\IId),\tilde{L}_q(\IId)).
\end{align*}	
Now Lemma \ref{lemma: widths} implies the lower bounds in \eqref{eq:dnWg}. 

We now prove \eqref{sampling-widths:p>2}. Assume  $q=2$.  The lower bound of \eqref{sampling-widths:p>2} follows from \eqref{eq-relations} and \eqref{eq:dnWg}. Let us verify the upper one.  By \eqref{eq:dnWg} we have that 
\begin{align}\label{d_n<}
	d_n
	\ll 
	n^{-\alpha} (\log n)^{(d-1)\alpha}.
\end{align}		
Notice that 
the separable normed space $\Wpgamma$ is continuously embedded into  $L_2(\RRd,\gamma)$, and the evaluation functional $f \mapsto f(\bx)$ is continuous on the space $\Wpgamma$ for each $\bx \in \RRd$. This means that $\BWpgamma$ satisfies Assumption A in \cite{DKU22}. By \cite[Corollary 4]{DKU22} and \eqref{d_n<} we prove the upper bound: 
$$
\varrho_n \ll d_n
\ll
	n^{-\alpha} (\log n)^{(d-1)\alpha}.
$$
\hfill 
\end{proof}

\subsection{The case $q = p = 2$}

Our  approach to this  case, which is completely different from the one in  the case $1 \leq q <p <\infty$, is similar to the hyperbolic cross trigonometric approximation  in the Hilbert space $\tilde{L}_2(\IId)$ of periodic functions from the Sobolev space $\tilde{W}_2^\alpha(\IId)$ (see, e.g., \cite{DTU18B} for details). Here, in the approximation, the trigonometric polynomials are replaced by the Hermite polynomials. 
 
For $k\in \NN_0$, the normalized probabilistic Hermite polynomial
$H_k$ of degree $k$ on $\RR$ is defined by
\begin{equation*} 
	H_k(x) 
	:= 
	\frac{(-1)^k}{\sqrt{k!}} 
	\exp\left(\frac{x^2}{2}\right) \frac{\rd^k}{\rd x^k} \exp\left(-\frac{x^2}{2}\right) .
\end{equation*}
For every multi-degree $\bk\in \NNd_0$, the $d$-variate Hermite
polynomial $H_\bk$ is defined by
\begin{equation*}\label{H_bk}
	H_\bk(\bx) :=\prod_{j=1}^d H_{k_j}(x_j),
	\;\; \bx\in \RRd.
\end{equation*}
It is well-known that the Hermite polynomials $\brab{H_\bk}_{\bk \in \NNd_0}$ constitute an orthonormal basis of the Hilbert space $L_2(\RRd,\gamma)$ (see, e.g.,  \cite[Section 5.5]{Szego1939}). In particular,  every $f \in L_2(\RRd,\gamma)$ can be represented by the Hermite series 
\begin{equation}\label{H-series}
	f = \sum_{\bk \in \NNd_0} \hat{f}(\bk) H_\bk \ \ {\rm with} \ \ \hat{f}(\bk) := \int_{\RRd} f(\bx)\, H_\bk(\bx)\gamma(\rd \bx) 
\end{equation}
converging in the norm of $L_2(\RRd,\gamma)$, and in addition, there holds  Parseval's identity
\begin{equation}\label{P-id}
	\norm{f}{L_2(\RRd,\gamma)}^2= \sum_{\bk \in \NNd_0} |\hat{f}(\bk)|^2.
\end{equation}

For $\alpha \in \NN_0$ and $\bk \in \NNd_0$, we define
\begin{equation*}\label{rho_bk}
	\rho_{\alpha,\bk}: = \prod_{j=1}^d \brac{k_j + 1}^\alpha.
\end{equation*}
\begin{lemma}\label{lemma:N-eq}
	Let $\alpha \in \NN_0$. Then we have that
	\begin{equation}\label{N-eq}
		\norm{f}{\Wa}^2 \asymp \sum_{\bk \in \NNd_0} \rho_{\alpha,\bk}|\hat{f}(\bk)|^2, \quad    f \in \Wa.
	\end{equation}
\end{lemma}
\begin{proof}
	This lemma in an implicit form has been proven in \cite[pages 687--688]{DILP18}. Let us prove it for completeness.
	From the formula for the $r$th derivative of the Hermite polynomial $H_k$
	\begin{equation*}\label{H-derivative}
		H_k^{(r)} = 
		\begin{cases}
			\sqrt{\frac{k!}{(k-r)!}}\,H_{k - r}, \ \ & {\rm if} \ k \ge r, \\
			0, & {\rm otherwise},
		\end{cases}
	\end{equation*}
	we deduce that  for $f \in W^\alpha_2(\RR,\gamma)$ and $r \le \alpha$,
	\begin{equation*}\label{D^r}
		f^{(r)} =  \sum_{k \ge r} \sqrt{\frac{k!}{(k-r)!}} \,\hat{f}(k)  H_{k-r}, 
	\end{equation*}
	and hence,
	\begin{equation}\label{NW-eq}
		\norm{f}{\Wa}^2 =  \sum_{r_1=0}^{\alpha}\sum_{k_1 \ge r_1} \frac{k_1!}{(k_1-r_1)!}\cdots \sum_{r_d=0}^{\alpha}\sum_{k_d \ge r_d}\frac{k_d!}{(k_d-r_d)!}
		|\hat{f}(k_1,\ldots,k_d)|^2.
	\end{equation}
	From the last equality and the relation $\frac{k!}{(k-r)!} \asymp \rho_{r,k}$, $ k \in \NN_0$, it is easy to derive \eqref{N-eq} for the case $d=1$. In the case $d \ge 2$, \eqref{N-eq} can be proven by induction on $d$ with the help of the equality \eqref{NW-eq}.
	\hfill
\end{proof}	

We extend the space $W^\alpha_2(\RRd,\gamma)$ to  any  $\alpha > 0$. Denote by $\Hh^\alpha$ the space of all   functions $f \in L_2(\RRd,\gamma)$ represented by the Hermite series \eqref{H-series} for which  the norm
	\begin{equation}\label{Hh-norm}
	\norm{f}{\Hh^\alpha} := \brac{\sum_{\bk \in \NNd_0} \rho_{\alpha,\bk}|\hat{f}(\bk)|^2}^{1/2}
\end{equation}
is finite.
With this definition, we identify $W^\alpha_2(\RRd,\gamma)$ with $\Hh^\alpha$ for $\alpha \in \NN$.

For  functions $f \in \Hh^\alpha$, we construct a hyperbolic cross approximation based on truncations of the Hermite series \eqref{H-series}. For the hyperbolic cross
	$G(\xi):= \brab{\bk \in \NNd_0: \rho_{1,\bk} \le \xi}, \ \ \xi  \in \RR_1,$
the truncation $S_\xi(f)$ of the Hermite series \eqref{H-series} on this set is defined by
\begin{equation*}\label{S_xi}
	S_\xi(f)	:= \sum_{\bk \in G(\xi)} \hat{f}(\bk) H_\bk.
\end{equation*}
Notice that $S_\xi$ is a linear projection from $L_2(\RRd,\gamma)$ onto the linear subspace $L(\xi)$ spanned by the Hermite polynomials $H_\bk$, $\bk \in G(\xi)$, and $\dim L(\xi) = |G(\xi)|$.

Recall that according to the section on notation in the introduction  $\boldsymbol{\Hh}^\alpha$ denotes the unit ball in $\Hh^\alpha$. 

\begin{theorem} 	\label{theorem:widths:p=q=2}
	Let $\alpha >0$. Then we can construct a sequence $\brab{\xi_n}_{n=2}^\infty$ with $|G(\xi_n)|\le n$ so that
	\begin{equation}\label{widths:p=q=2}
		\sup_{f\in\boldsymbol{\Hh}^\alpha} \norm{f - S_{\xi_n}(f)}{L_2(\RRd,\gamma)}
		\asymp
		\lambda_n(\boldsymbol{\Hh}^\alpha, L_2(\RRd,\gamma)) 
		=
		d_n(\boldsymbol{\Hh}^\alpha, L_2(\RRd,\gamma))
		\asymp 
		n^{-\frac{\alpha}{2}} (\log n)^{\frac{(d-1)\alpha}{2}}. 
	\end{equation}
Moreover, with the additional condition $\alpha > 1$,
\begin{align}\label{sampling-widths:p=q=2}
	\varrho_n(\boldsymbol{\Hh}^\alpha, L_2(\RRd,\gamma))
	\asymp 
	n^{-\frac{\alpha}{2}} (\log n)^{\frac{(d-1)\alpha}{2}}.
\end{align}		
\end{theorem}
\begin{proof}
Since 	$L_2(\RRd,\gamma)$ is a Hilbert space, we have the equality 
$\lambda_n(\boldsymbol{\Hh}^\alpha, L_2(\RRd,\gamma)) 
=
d_n(\boldsymbol{\Hh}^\alpha, L_2(\RRd,\gamma))$ in \eqref{widths:p=q=2}.
To prove the upper bounds in \eqref{widths:p=q=2} it is sufficient to construct a sequence $\brab{\xi_n}_{n=2}^\infty$ so that $|G(\xi_n)|\le n$ and
	\begin{align}\label{f-S_xi(f)}
		\sup_{f\in \boldsymbol{\Hh}^\alpha} \norm{f - S_{\xi_n}(f)}{L_2(\RRd,\gamma)}
		\ll
		n^{-\frac{\alpha}{2}} (\log n)^{\frac{(d-1)\alpha}{2}}.
	\end{align}		
	From Parseval's identity \eqref{P-id} and Lemma \ref{lemma:N-eq} we have that for every $f \in \bW^\alpha_2(\RRd,\gamma)$ and $\xi > 1$,
	\begin{align}\label{f-S_xi(f)2}
		\norm{f - S_{\xi}(f)}{L_2(\RRd,\gamma)}^2
		=
		\sum_{\bk \notin G(\xi)} \hat{f}(\bk)^2
		\ll
		\xi^{-\alpha}	\sum_{\bk \notin G(\xi)} \rho_{\alpha,\bk}\hat{f}(\bk)^2
		\ll      
		\xi^{-\alpha} \norm{f}{W^\alpha_2(\RRd,\gamma)} \le \xi^{-\alpha} .
	\end{align}
	Let 	$\brab{\xi_n}_{n=2}^\infty$ be the sequence of $\xi_n$ defined as the largest number satisfying the condition 	$|G(\xi_n)|\le n$. From the relation $|G(\xi_n)| \asymp \xi_n (\log \xi_n)^{d-1}$, see, e.g., \cite[page 130]{Tem93B}, we derive that $\xi_n^{-\alpha} \asymp n^{-\alpha} (\log n)^{(d-1)\alpha}$ which together with \eqref{f-S_xi(f)2} yields \eqref{f-S_xi(f)}.
	
	 To show the lower bounds of \eqref{widths:p=q=2} we need 
	Tikhomirov's theorem~\cite[Theorem 1]{Tikho60} which states  that if $X$ is a Banach space and $U_{n+1}(\lambda)$ the ball of radius $\lambda >0$ in a linear $n+1$-dimensional subspace of $X$, then $d_n(U_{n+1}(\lambda),X)=\lambda$. Further, if  
	$$U(\xi):=\brab{f \in L(\xi): \norm{f }{L_2(\RRd,\gamma)} \le 1}$$ and $f \in U(\xi)$, then by Parseval's identity \eqref{P-id} and the definition of $\boldsymbol{\Hh}^\alpha$, similarly to \eqref{f-S_xi(f)2}, we deduce that
	$\norm{f}{\Hh^\alpha}	\ll \xi^{\alpha/2}$.  This means that  
	$C \xi^{\alpha/2} U(\xi) \subset \boldsymbol{\Hh}^\alpha$ for some $C > 0$. Let 	$\brab{\xi'_n}_{n=2}^\infty$ be the sequence of $\xi'_n$ defined as the smallest number satisfying the condition 	$|G(\xi'_n)|\ge n + 1$. Then $\dim L(\xi'_n) = |G(\xi'_n)|\ge n + 1$, and
	similarly as in the upper estimation,  
	$(\xi'_n)^{-\alpha} \asymp n^{-\alpha} (\log n)^{(d-1)\alpha}$. 
	By Tikhomirov's theorem  for the smallest quantity 
	$d_n$ in \eqref{widths:p=q=2} we have that
	\begin{align*}
		d_n(\boldsymbol{\Hh}^\alpha, L_2(\RRd,\gamma))
		\ge
		d_n(C\xi^{\alpha/2} U(\xi'_{n+1}) ,L_2(\RRd,\gamma))
		\gg
		(\xi'_n)^{-\alpha}
		\asymp 
		n^{-\frac{\alpha}{2}} (\log n)^{\frac{(d-1)\alpha}{2}} .
	\end{align*}

Let us prove \eqref{sampling-widths:p=q=2}. The lower bound of \eqref{sampling-widths:p=q=2} follows from \eqref{widths:p=q=2} and the inequality 
$\varrho_n(\boldsymbol{\Hh}^\alpha, L_2(\RRd,\gamma)) 
\ge
\lambda_n(\boldsymbol{\Hh}^\alpha, L_2(\RRd,\gamma))$.
We verify the upper one. By \eqref{widths:p=q=2}, 
\begin{align} \label{d_n<2}
	d_n(\boldsymbol{\Hh}^\alpha, L_2(\RRd,\gamma))
	\ll
	n^{-\frac{\alpha}{2}} (\log n)^{\frac{(d-1)\alpha}{2}}.
\end{align} 
Notice that for $\alpha >1$, $\Hh^\alpha$ is a separable reproducing kernel Hilbert space with the reproducing kernel
\begin{equation}\label{RK}
	K(\bx,\by)= \sum_{\bk \in \NNd_0} \rho_{\alpha,\bk}^{-1} H_\bk(\bx)H_\bk(\by).
\end{equation}
From the orthonormality of the system $\brab{H_\bk}_{\bk \in \NNd_0}$ it is easily seen that $	K(\bx,\by)$ satisfies the finite trace assumption
\begin{equation}\label{trace assumption}
\int_{\RRd} K(\bx,\bx)\gamma(\rd \bx)  \ < \ \infty.
\end{equation}
Hence  by \cite[Corollary 2]{DKU22} we obtain $\varrho_n(\boldsymbol{\Hh}^\alpha, L_2(\RRd,\gamma)) 
\ll
d_n(\boldsymbol{\Hh}^\alpha, L_2(\RRd,\gamma))$. This and \eqref{d_n<2} prove the upper bound of \eqref{sampling-widths:p=q=2}.
	\hfill	
\end{proof}

In the case when $\alpha \in \NN$, Theorem \ref{theorem:widths:p=q=2} yields the following result on sampling $n$-widths of  the Sobolev class $\bW^\alpha_2(\RRd,\gamma)$ of mixed smoothness $\alpha$.
	
\begin{corollary} 	\label{corollary:widths:p=q=2}
	Let $\alpha\in \NN$. Then we can construct a sequence $\brab{\xi_n}_{n=2}^\infty$ with $|G(\xi_n)|\le n$ so that
	\begin{align}\label{Cor-widths:p=q=2}
		\sup_{f\in\bW^\alpha_2(\RRd,\gamma)} \norm{f - S_{\xi_n}(f)}{L_2(\RRd,\gamma)}
		\asymp
		\lambda_n 
		=
		d_n
		\asymp 
		n^{-\frac{\alpha}{2}} (\log n)^{\frac{(d-1)\alpha}{2}}.
	\end{align}		
Moreover, with the additional condition $\alpha \ge 2$,
\begin{align}\label{W-sampling-widths:p=q=2}
	\varrho_n
	\asymp 
	n^{-\frac{\alpha}{2}} (\log n)^{\frac{(d-1)\alpha}{2}}.
\end{align}		
\end{corollary}

We stress that the assumption $\alpha > 1$  for \eqref{sampling-widths:p=q=2}  is vital since  it is a necessary and sufficient condition for $\Hh^\alpha$ to be  a separable reproducing kernel Hilbert space 
with the finite trace condition \eqref{trace assumption} and therefore, the result \cite[Corollary 2]{DKU22} can be applied. We conjecture that the consequent asymptotic order \eqref{W-sampling-widths:p=q=2}  still holds true for $\alpha = 1$. Here it may require a different technique.

\section{Numerical comparison with other quadratures}\label{Sec-4}
We  illustrate the integration nodes of the quadratures constructed in the present paper, in comparison with the integration nodes used in \cite{DILP18}. Assume that $\{\bx_1,\ldots,\bx_n\}$ are the integration nodes for  an optimal quadrature $I_n$  for functions in ${W}^\alpha_p(\II^2)$. Then the integration nodes in \cite{DILP18} are just a dilation of these nodes to the cube $[-C\sqrt{\log n}, C\sqrt{\log n}]^2$. Hence these nodes are distributed similarly on this cube. Differently,  the integration nodes in our construction  are  formed from certain integer-shifted dilations of $\{\bx_1,\ldots,\bx_m\}$ and contained in the ball of radius $C\sqrt{\log n}$. These nodes are dense when they are near the origin and getting sparser as they are farther from the origin. The illustration is given in Figure \ref{Fig-1}.
\begin{figure}
	\begin{tabular}{cc}
		\includegraphics[height=7.5cm]{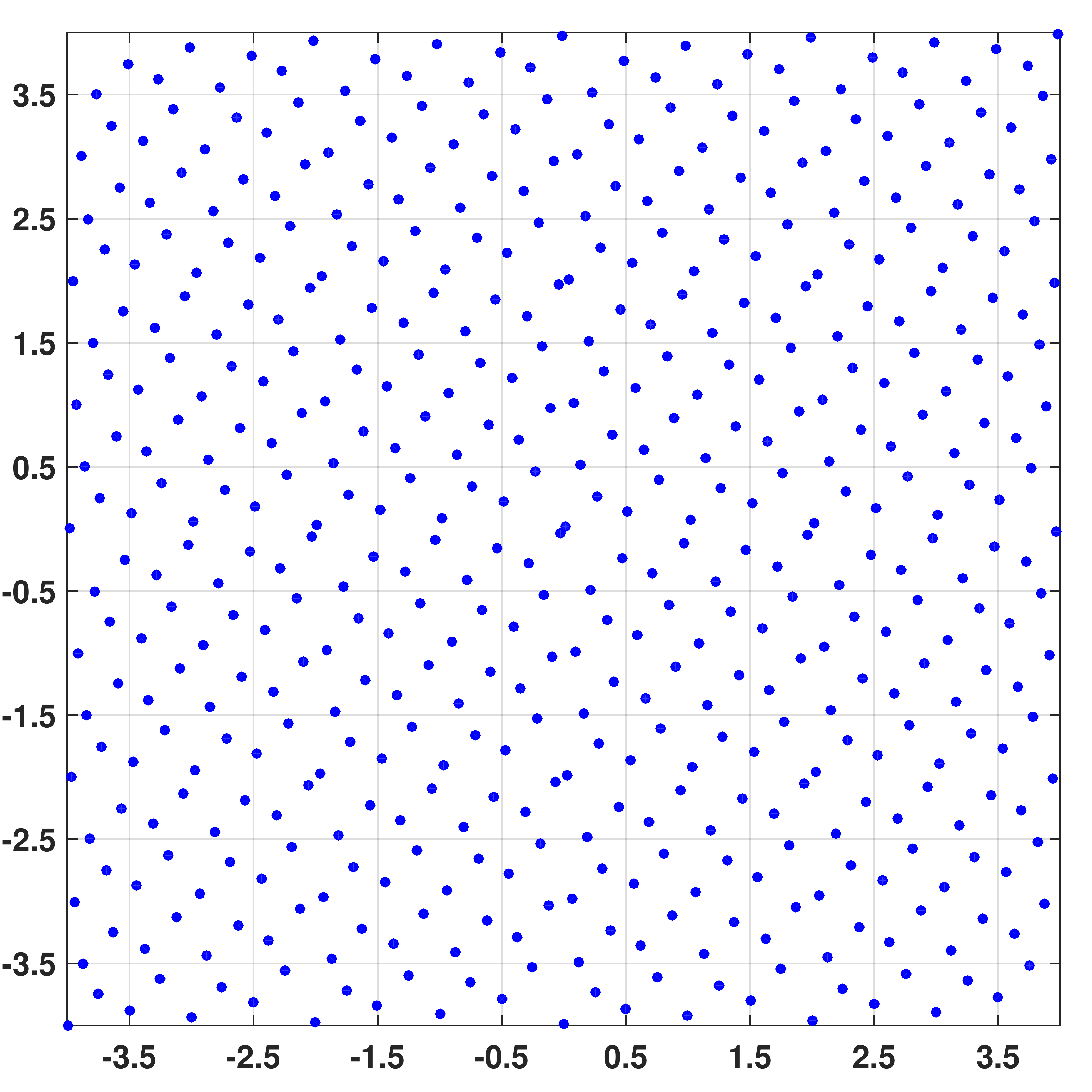}	 &  \includegraphics[height=7.5cm]{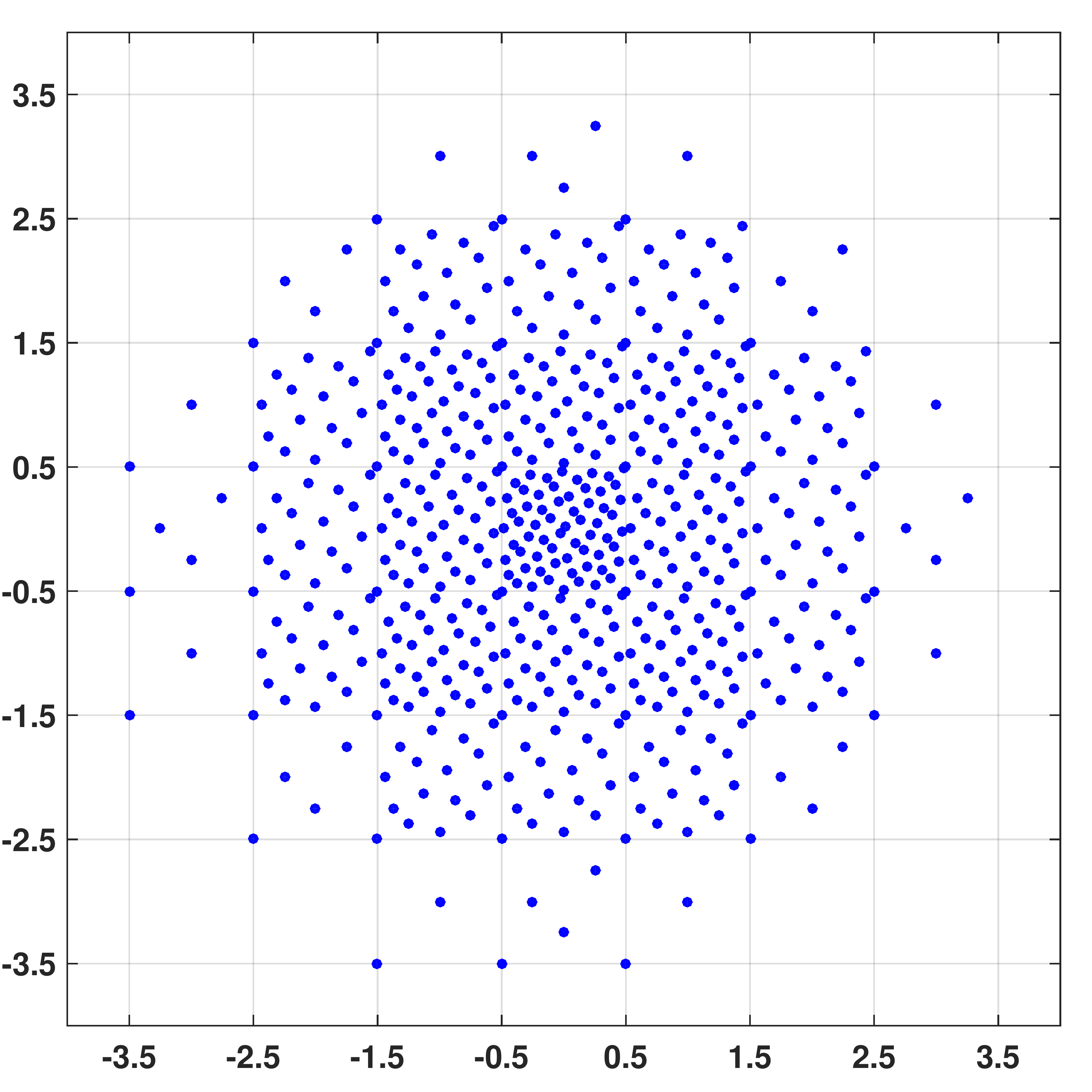}
		\\
		{Point set of Dick et al. \cite{DILP18} (512 points)} & {Point set of our construction (560 points)}
	\end{tabular}
\caption{Distribution of integration nodes in \cite{DILP18} and in this paper.}
\label{Fig-1}
\end{figure}

The following is a numerical test of our result for the cases $d=1$ and $\alpha = 1,2,3$. We consider the algorithm for the space ${W}^\alpha_2(\RR)$. For numerical integration of  functions in $\mathring{W}^\alpha_2(\II)$ we use the Smolyak point set. Observe that these nodes give the optimal convergence rate since $d=1$, see Section \ref{sec-optimal}. In this test $\delta$ in \eqref{[tau]<e^{-delta k}1}  is chosen as $\delta=\frac{1}{6}$. We apply  the method of change of variable by $\psi_3$ to get asymptotically optimal integration nodes and weights for functions in $W^\alpha_2(\II)$  where the function $\psi_3$ is defined as in \eqref{psin}. From these nodes and weights we get the optimal quadrature $\{x_1,\ldots,x_n\}$ and $\{\lambda_1,\ldots,\lambda_n\}$ for $W^\alpha_2(\RR,\gamma)$ as described in Section \ref{Subsec-AssemblingQuadratures}. The error of this quadrature  is given by
$$
{{\rm err}}=\Bigg(\bigg(1-\sum_{i=1}^n \lambda_i\bigg)^2+\sum_{k=1}^\infty \rho_{\alpha,k}\bigg(\sum_{i=1}^n\lambda_iH_k(x_i)\bigg)^2\Bigg)^{1/2},
$$
see, e.g., \cite[Section 4]{DILP18}.

For the numerical computation this error is replaced by the  truncated version
$$
{{\rm err}}_m=\Bigg(\bigg(1-\sum_{i=1}^n \lambda_i\bigg)^2+\sum_{k=1}^{m} \rho_{\alpha,k}\bigg(\sum_{i=1}^n\lambda_iH_k(x_i)\bigg)^2\Bigg)^{1/2}.
$$
In our test we choose $m=10^5$. Our result is given in Figure \ref{Fig-2} which shows that the worst-case errors of the assembled quadratures for $\alpha\in \{1,2,3\}$ have convergence rate $\mathcal{O}(n^{-\alpha})$. It has been observed in \cite{DILP18} that the interlaced  Sobol' sequence also gives the optimal convergence rates for numerical integration of $W^\alpha_2(\RR,\gamma)$. The numerical result reaffirms the theory in this paper.

\begin{figure}
\includegraphics[height=8.5cm]{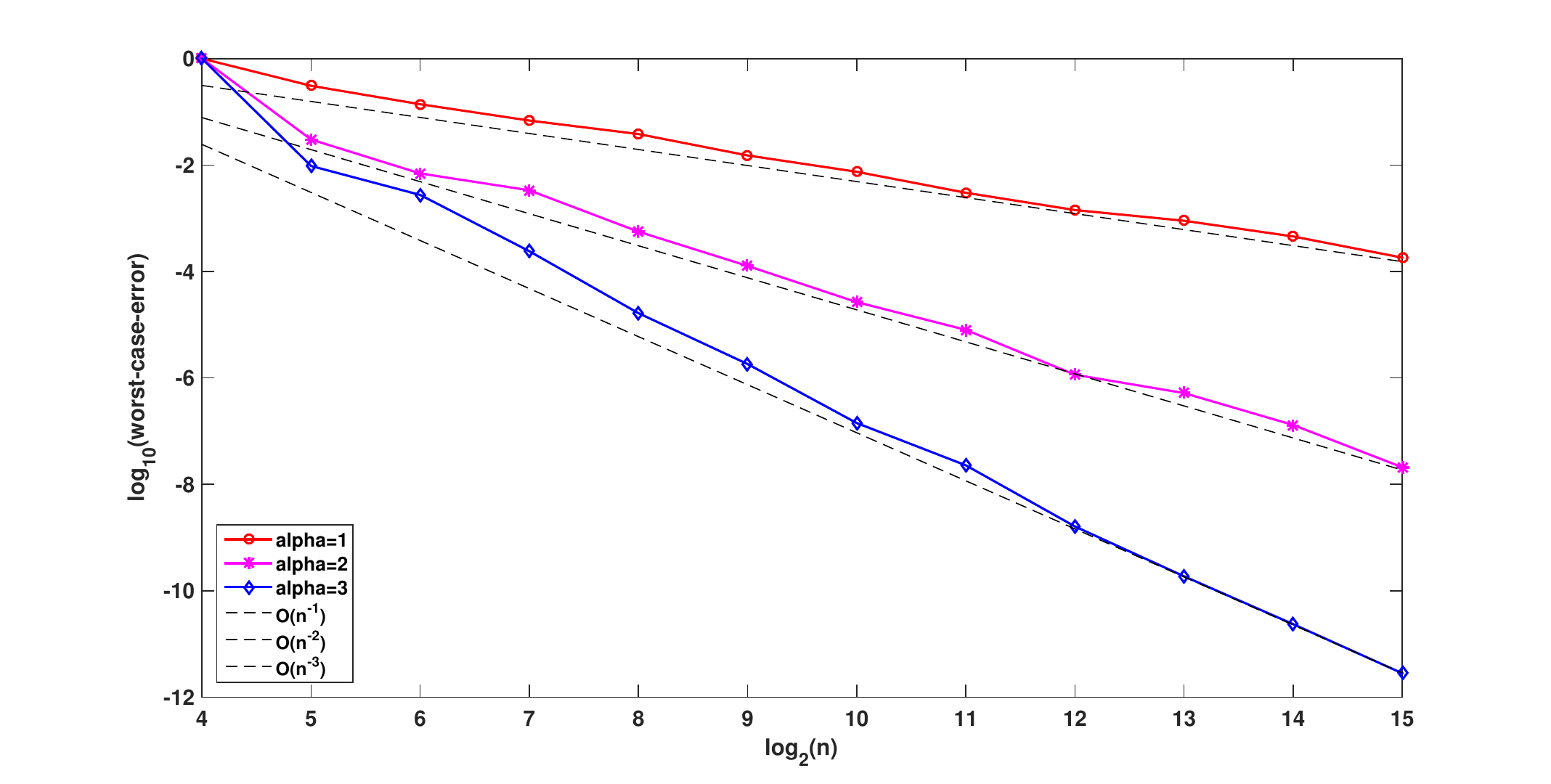}
\caption{Errors of the assembling quadratures.}
\label{Fig-2}
\end{figure}
\noindent
{\bf Acknowledgments:} This research is funded by Vietnam Ministry of Education and Training under Grant No. B2023-CTT-08. A part of this work was done when  the authors were working at the Vietnam Institute for Advanced Study in Mathematics (VIASM). They would like to thank  the VIASM for providing a fruitful research environment and working condition. The authors express special thanks to David Krieg and Mario Ullrich for useful discussions, in particular, for pointing out the recent paper \cite{DKU22} and for suggesting to include the results on sampling $n$-widths into the present paper. They express gratitude to the referees for valuable comments and suggestions which improved the presentation of this paper. 
\bibliographystyle{abbrv}

\bibliography{AllBib}
\end{document}